\documentclass[11pt,a4paper]{article}
\usepackage{amsmath}
\usepackage{color}
\usepackage[utf8]{inputenc}
\usepackage{tikz}
\usepackage{multirow,cancel}
\usepackage{float}
\usepackage[shortlabels]{enumitem}
\usepackage[export]{adjustbox}
\RequirePackage[colorlinks,citecolor=blue,urlcolor=blue,linkcolor=blue]{hyperref}
\usepackage{hyperref}
 \usepackage{authblk}
\usepackage[disable,textwidth=4cm, textsize=footnotesize]{todonotes}%

\usepackage{hyperref,aliascnt}
\usepackage{latexsym,euscript}
\usepackage{amsbsy,amscd,amsfonts,amsmath,amsopn,amssymb,amstext,amsthm,amsxtra,enumitem,bbm}

\usepackage{epsf,epsfig,graphics, color}
\usepackage{ float, xargs, pdfsync}
\usepackage{appendix}

\usepackage{fancybox,ifthen,twoopt}
\newtheorem{theorem}{Theorem}
\newaliascnt{proposition}{theorem}

\aliascntresetthe{proposition}

\newaliascnt{lemma}{theorem}
\newtheorem{lemma}[lemma]{Lemma}
\aliascntresetthe{lemma}

\theoremstyle{definition}
\newtheorem{example}{Example}

\newtheorem{remark}{Remark}

\newaliascnt{corollary}{theorem}

\aliascntresetthe{corollary}

\newaliascnt{definition}{theorem}
\newtheorem{definition}[definition]{Definition}
\aliascntresetthe{definition}

\newcounter{hypA'}

\newenvironment{hyp}[1]{
\begin{enumerate}[label=(\textbf{\sf #1}-\arabic*),resume=hyp#1]\begin{sf}}
{\end{sf}\end{enumerate}}

\def\rme{\mathrm{e}}
\def\Cset{\mathsf{C}}

\def\Xset{\mathsf{X}}
\def\Yset{\mathsf{Y}}

\def\Xsigma{\mathcal{X}}
\def\Ysigma{\mathcal{Y}}

\def\met{\Delta}
\def\rset{\ensuremath{\mathbb{R}}}
\def\nset{\ensuremath{\mathbb{N}}}
\def\zset{\ensuremath{\mathbb{Z}}}

\def\mcf{\mathcal{F}}

\def\eqsp{\;}
\def\PE{\mathbb{E}}

\def\PP{\mathbb{P}}

\def\rmd{\mathrm{d}}
\newcommandx\LogInt[5][1=\theta,4=,5=Y]{\upsilon_{#4}^{#1}\langle {#5}_{#2:#3} \rangle}
\newcommand{\ball}[2]{\mathsf{B}(#1, #2)}

\newcommand{\as}{\mbox{-a.s.}}

\newcommand{\ie}{i.e.}

\newcommandx{\aslim}[1]{\ensuremath{\stackrel{#1-\text{a.s.}}{\longrightarrow}}}

\newcommandx\sequence[3][2=n,3=\zset]{\ensuremath{\{ #1_{#2} \}_{#2 \in #3}}}

\newcommand{\CPE}[3][]
{\ifthenelse{\equal{#1}{}}%
{\mathbb{E}\left[\left. #2 \, \right| #3 \right]}
{\mathbb{E}_{#1}\left[\left. #2 \, \right| #3 \right]}
}
\newcommand{\CPEu}[4][]
{\ifthenelse{\equal{#1}{}}%
{\mathbb{E}\left[\left. #2 \, \right| #3 \right]}
{\mathbb{E}^{#1}_{#2}\left[\left. #3 \, \right| #4 \right]}
}

\newcommand{\CPEv}[3][]
{\ifthenelse{\equal{#1}{}}%
{\mathbb{E}_\star\left[\left. #2 \, \right| #3 \right]}
{\mathbb{E}_\star_{#1}\left[\left. #2 \, \right| #3 \right]}
}

\newcommand{\CPP}[3][]
{\ifthenelse{\equal{#1}{}}%
{\mathbb{P}\left[\left. #2 \, \right| #3 \right]}
{\mathbb{P}_{#1}\left[\left. #2 \, \right| #3 \right]}
}

\newcommand{\CPPu}[3][]
{\ifthenelse{\equal{#1}{}}%
{\mathbb{P}\left[\left. #2 \, \right| #3 \right]}
{\mathbb{P}^{#1}\left[\left. #2 \, \right| #3 \right]}
}

\newcommandx{\chunk}[3]%
{\ensuremath{#1}_{#2:#3}}

\def\1{\mathbbm{1}}

\newcommandx\proj[2][1=,2=]{
\ifthenelse{\equal{#1}{}}
{\operatorname{X}}
{\operatorname{X}_{#1:#2}}i
}

\newcommand{\eqdef}{:=}

\newcommandx\lkdM[3][1=,3=]{
\ifthenelse{\equal{#2}{}}
{ \mathsf{L}_{#1}^{#3}}
{ \mathsf{L}_{#1}^{#3}\langle #2\rangle}
}
\newcommandx\lkdMStat[3][1=,3=]{
\ifthenelse{\equal{#2}{}}
{ \bar{\mathsf{L}}_{#1}^{#3}}
{ \bar{\mathsf{L}}_{#1}^{#3}\langle #2 \rangle}
}

\newcommandx\lkd[3][1=,3=]{
\ifthenelse{\equal{#2}{}}
{ \ell_{#1}^{#3}}
{ \ell_{#1}^{#3}\langle #2\rangle}
}
\newcommandx\lkdStat[3][1=,3=]{
\ifthenelse{\equal{#2}{}}
{ \bar \ell_{#1}^{#3}}
{ \bar \ell_{#1}^{#3}\langle #2 \rangle}
}

\newcommand{\mlY}[1]{\hat\theta_{#1}}
\newcommand{\argmax}{\mathop{\mathrm{argmax}}}

\newcommandx{\normLip}[2][1=]{\mathrm{Lip}(#2;#1)}
\newcommandx{\wass}[2][1]{\lVert #2\rVert_{#1}}

\newcommandx{\wasser}[3][1=]{\mathcal{W}_{#1}\left(#2,#3\right)}
\newcommandx{\proho}[3][1=]{\mathcal{P}_{#1}\left(#2,#3\right)}

\newcommandx{\dobru}[2][1=]{\dobrush_{#1}\left( #2\right)}
\newcommand{\dobrush}{\Delta}

\newcommandx{\Pcan}[2][1=,2=]{\mathbb{P}_{#1}^{#2}}
\newcommandx{\Ecan}[2][1=,2=]{\mathbb{E}_{#1}^{#2}}

\newcommandx\cesp[4][1=,2=]{\ensuremath{{\mathbb E}_{#1}^{#2}\left[ \left. #3 \right| #4 \right]}}

\newcommandx{\f}[2][1=\theta]{\psi^{#1}\langle #2 \rangle}
\newcommand{\thv}{{\theta_\star}}
\newcommandx{\kap}[3][1=\theta]{
\ifthenelse{\equal{#3}{}}
{\kappa^{#1}\langle #2\rangle }
{\kappa^{#1}\langle #2\rangle \left(#3\right)}
}

\def\lnp{\ln^{+}}

\def\zsetp{\zset_{+}}
\def\zsetn{\zset_{-}}
\def\zsetpnz{\zset_{+}^*}

\def\rsetp{\rset_{+}}

\newcommandx{\probdoeblin}[3][1=]{\mu^{#1}_{#2}\langle #3 \rangle}
\newcommand{\Pblock}[2][]
{\ifthenelse{\equal{#1}{}}{\boldsymbol{\operatorname{L}}\langle#2\rangle}{\boldsymbol{\operatorname{L}}^{#1}\langle#2\rangle}
}

\newcommand{\ConPblock}[3][]
{\ifthenelse{\equal{#1}{}}{\boldsymbol{\operatorname{L}}\langle#2|#3\rangle}{\boldsymbol{\operatorname{L}}^{#1}\langle#2|#3\rangle}
}

\newcommand{\pblock}[2][]
{\ifthenelse{\equal{#1}{}}{\mathbf{\ell}\langle#2\rangle}{\mathbf{\ell}^{#1}\langle #2\rangle}
}

\newcommand{\Sset}{\mathsf{S}}
\newcommand{\Ssigma}{\mathcal{S}}

\def\simplex{\mathsf{P}}

\newcommandx{\limlike}[4][1=\theta, 2=\theta_\star]{p^{#1,#2}\left( #3\mid #4 \right)}

\def\radius{|\lambda|_{\max}}

\def\Xmet{\rmd_\Xset}

\def\Xmet{\rmd}

\begin{document}




\title{
Handy sufficient conditions for the convergence
 of the maximum likelihood estimator in observation-driven models}

\author[1]{Randal Douc\thanks{randal.douc@telecom-sudparis.eu}}
\author[2]{Fran\c{c}ois Roueff\thanks{roueff@telecom-paristech.fr}}
\author[2]{Tepmony Sim\thanks{sim@telecom-paristech.fr}}
\affil[1]{Department  CITI, CNRS UMR 5157, Telecom Sudparis, Evry, France.}
\affil[2]{Institut Mines-Telecom, Telecom Paristech, CNRS LTCI, Paris, France.}

\renewcommand\Authands{ and }

%
%
%
%

\date{April 7, 2015}



\sloppy
\maketitle

\begin{abstract}
  This paper generalizes asymptotic properties obtained in the
  observation-driven times series models considered by \cite{dou:kou:mou:2013}
  in the sense that the conditional law of each observation is also permitted
  to depend on the parameter. The existence of ergodic solutions and the
  consistency of the Maximum Likelihood Estimator (MLE) are derived under
  easy-to-check conditions. The obtained conditions appear to apply for a wide
  class of models. We illustrate our results with specific observation-driven
  times series, including the recently introduced NBIN-GARCH and NM-GARCH
  models, demonstrating the consistency of the MLE for these two models.
\end{abstract}

\section{Introduction}
Observation-driven time series models have been widely used in various
disciplines such as in economics, finance, epidemiology, population dynamics,
etc. These models have been introduced by \cite{cox:1981} and later considered
by \cite{streett:2000}, \cite{davis:dunsmuir:streett:2003},
\cite{fokianos:rahbek:tjostheim:2009}, \cite{neuman:2011},
\cite{doukhan:fokianos:tjostheim:2012}, \cite{davis:liu:2012} and
\cite{dou:kou:mou:2013}. The celebrated GARCH$(1,1)$ model, see
\cite{bollerslev:1986}, as well as most of the models derived from this one,
see \cite{bollerslev08-glossary} for a list of some of them, are typical
examples of observation-driven models. Observation-driven models have the nice
feature that the associated (conditional) likelihood and its derivatives are
easy to compute and the prediction is straightforward. The consistency of the
maximum likelihood estimator (in short, MLE) for the class of these models can
be cumbersome, except when it can be derived using computations specific to the
studied model (the GARCH(1,1) case being one of the most celebrated
example). When the observed variable is discrete, general consistency results
have been obtained only recently in \cite{davis:liu:2012} or
\cite{dou:kou:mou:2013} (see also in \cite{henderson:matteson:woodard:2011} for
the existence of stationary and ergodic solutions to some observation-driven
time series models). However, the consistency result of \cite{dou:kou:mou:2013}
applies to some restricted class of models and does not cover the case where
the distribution of the observations given the hidden variable also depends on
an unknown parameter.  We now introduce three simple examples, to which the
results of \cite{dou:kou:mou:2013} can not be directly applied.  The first one
is the negative binomial integer-valued GARCH (NBIN-GARCH) model, which was
first introduced by \cite{zhu:2011} as a generalization of the Poisson IN-GARCH
model. The NBIN-GARCH model belongs to the class of integer-valued GARCH models
that account for overdispersion (\ie, variability is larger than mean) and
potential heavy tails in the high values. In \cite{zhu:2011}, the author
applied this model to treat the data of counts of poliomyelitis cases in the
USA from 1970 to 1983 reported by the Centres for Disease Control, where data
overdispersion was detected. The estimation result showed that
NBIN-GARCH$(1,1)$ outperformed among some commonly used models such as Poisson
and Double Poisson models. The NBIN-GARCH$(1,1)$ model is formally defined as
follows.
\begin{example}[NBIN-GARCH$(1,1)$ model]\label{example:Nbigarch:defi}
Consider the following recursion.
\begin{align}\label{eq:def:nbingarch}\begin{split}
&X_{k+1}=\omega+a X_k+bY_{k}\eqsp, \\
&Y_ {k+1}|\chunk{X}{0}{k+1},\chunk{Y}{0}{k}\sim \mathcal{NB} \left(r,\frac{X_{k+1}}{1+X_{k+1}}\right)\eqsp,
\end{split}
\end{align}
where $X_k$ takes values in $\Xset=\rsetp$, $Y_k$ takes values in $\zsetp$ and
$\theta=(\omega, a, b,r)\in(0,\infty)^4$ is an unknown parameter. In
\eqref{eq:def:nbingarch}, $\mathcal{NB}(r, p)$ denotes the negative binomial
distribution with parameters $r>0$ and $p\in (0,1)$, that is: if $Y \sim
\mathcal{NB}(r, p)$, then $\PP(Y=k)=\frac{\Gamma(k+r)}{k!\Gamma(r)} (1-p)^r
p^{k}$ for all $k\geq0$, where $\Gamma$ stands for the Gamma function.  Though
substantial analysis on this model has been carried out in the literature, to
the best of our knowledge, the consistency of the MLE has not been treated, see
the end of the discussions of Section~6 in \cite{zhu:2011}.
\end{example}
The second example is the univariate normal mixture GARCH model (NM-GARCH)
proposed by \cite{haas2004mixed} and later considered by
\cite{alexander2006normal}. The NM-GARCH model is another natural extension of
GARCH processes, where the usual Gaussian conditional distribution of the
observations given the hidden volatility variable is replaced by a mixture of
Gaussian distributions given a hidden vector volatility variable. The NM-GARCH
model has the ability of capturing time variation in both conditional skewness
and kurtosis, while the classical GARCH cannot. In \cite{alexander2006normal},
the NM-GARCH$(1,1)$ model was applied to examine the data of exchange rates
consisting of daily prices in US dollars of three different currencies (British
pound, euro and Japanese yen) from 2 January 1989 to 31 December 2002. The
empirical evidence suggested the best performance of NM$(2)$-GARCH$(1,1)$ when
compared to the classical GARCH$(1,1)$, standardized symmetric and skewed
$t$-GARCH$(1,1)$ models applied to this same data. The definition of this model
is formally stated as follows.
\begin{example}[NM$(d)$-GARCH$(1,1)$ model]
\label{example:nmgarch:defi}
Let $d \in \nset\setminus \{0\}$ and consider the following recursion.
\begin{align}
&\mathbf{X}_{k+1}=\boldsymbol{\omega}+\mathbf{A}\mathbf{X}_k+Y_{k}^2\mathbf{b} \eqsp,\nonumber \\
&Y_ {k+1}|\chunk{\mathbf{X}}{0}{k+1},\chunk{Y}{0}{k}\sim G^\theta(\mathbf{X}_{k+1};\cdot) \eqsp, \label{eq:mmdm:def}\\
&\frac{\rmd G^\theta(\mathbf{x};\cdot)}{\rmd\nu}(y)= \sum_{\ell=1}^{d}{\gamma_\ell\frac{\rme^{-y^2/{2x_{\ell}}}}{(2\pi x_{\ell})^{1/2}}}\eqsp, \quad \mathbf{x}\in(0,\infty)^d,\;y\in\rset\eqsp,  \nonumber
\end{align}
where $\nu$ is the Lebesgue measure on $\rset$, $\mathbf{X}_k=[X_{1,k} \ldots X_{d,k}]^T$ takes values in  $\Xset=\rsetp^d$;
$\boldsymbol{\gamma}=[\gamma_1\ldots\gamma_d]^T$ a $d$-dimensional vector of
mixture coefficients belonging to the $d$-dimensional simplex
\begin{equation}
  \label{eq:simplex-def}
  \simplex_d=\left\{\boldsymbol{\gamma}\in\rsetp^d~:~
  \sum_{\ell=1}^{d}{\gamma_\ell}=1\right\}\;,
\end{equation}
$\boldsymbol{\omega}$, $\mathbf{b}$ are $d$-dimensional vector parameters with
positive and non-negative entries, respectively and $\mathbf{A}$ is a $d\times
d$ matrix parameter with non-negative entries. Here we have
$\theta=\left(\boldsymbol{\gamma}, \boldsymbol{\omega}, \mathbf{A},
  \mathbf{b}\right)$. Note that $G^\theta$ depends on $\theta$ only through the
mixture coefficients $\gamma_1,\dots,\gamma_d$. If $d=1$, we obtain the usual
conditionally Gaussian GARCH(1,1) process. In such a case, since
$\boldsymbol{\gamma}=\gamma_1=1$, $G^\theta$ no longer depends on $\theta$.  Up
to our knowledge, the usual consistency proof of the MLE for the GARCH cannot
be directly adapted to this model.
\end{example}
Finally, we consider the following new example, where a threshold is added to the
usual INGARCH model in the conditional distribution.
\begin{example}[Threshold INGARCH model]
\label{example:tingarch}
Consider the following recursion.
\begin{align}\label{eq:def:tingarch}\begin{split}
&X_{k+1}=\omega+a X_k+bY_{k}\eqsp, \\
&Y_ {k+1}|\chunk{X}{0}{k+1},\chunk{Y}{0}{k}\sim \mathcal{P} \left(X_{k+1}\wedge \tau\right)\eqsp,
\end{split}
\end{align}
where $X_k$ takes values in $\Xset=(0, \infty)$, $Y_k$ takes values in $\zsetp$
and $\theta=(\omega, a, b, \tau)\in(0,\infty)^4$ is an unknown
parameter. Comparing with the usual INGARCH model, a threshold $\tau$ has been
added in the conditional observation distribution. This corresponds to the
practical case where the hidden variable has an influence on the observation up
to this threshold.
\end{example}

For a well-specified model, a classical approach to establish the consistency
of the MLE generally involves two main steps: first the maximum likelihood
estimator (MLE) converges to the maximizing set $\Theta_\star$ of a limit
criterion, and second the maximizing set indeed reduces to the true parameter
$\thv$, which is usually referred to as solving the \emph{identifiability}
problem. In this paper, we are interested in solving the problem involved in
the first step, that is, the convergence of MLE.  We extend the convergence
result of MLE obtained in \cite{dou:kou:mou:2013}, which is valid for a
restricted class of models, to a larger class of models in which the three
examples introduced above are embedded. More precisely, we show the convergence
of MLE in observation-driven models where the probability distributions of
observations explicitly depend on the unknown parameters. Moreover, we provide
very simple conditions that are easy to check, as shown by the three
illustrating examples.

The paper is organized as follows. Specific definitions and notation are
introduced in \autoref{sec:definitions-notation-od}. Then,
\autoref{sec:main-result} contains the main contribution of the paper, that is,
sufficient conditions for the existence of ergodic solutions and for the
consistency of the MLE. These results are then applied in
\autoref{sec:examples} to the three examples introduced above.
Numerical experiments for the NBIN-GARCH$(1,1)$ model are
  given in \autoref{sec:numer-exper}. Finally, \autoref{sec:proofs} provides
the proofs of the main results, mainly inspired
from~\cite{dou:kou:mou:2013}.

\section{Definitions and notation}
\label{sec:definitions-notation-od}
Consider a bivariate stochastic process $\{(X_k,Y_k)\,:\,k\in\zsetp\}$ on
$\Xset\times\Yset$, where $(\Xset,\Xmet)$ is a complete and separable metric space
endowed with the associated Borel $\sigma$-field $\Xsigma$ and
$(\Yset,\Ysigma)$ is a Borel space. Let $(\Theta,\met)$, the set of parameters,
be a compact metric space, $\{G^\theta\,:\,\theta\in\Theta\}$ be a family of
probability kernels on $\Xset \times \Ysigma$ and $\{(x,y)
\mapsto \psi^\theta_y(x)\,:\, \theta \in \Theta\}$ be a family of measurable
functions from $(\Xset\times \Yset, \Xsigma \otimes \Ysigma)$ to $(\Xset,
\Xsigma)$. The observation-driven time series model can be formally defined as
follows.

\begin{definition}\label{def:obs-driv:gen}
  A time series $\{Y_k \,:\,k\in\zsetp\}$ valued in $\Yset$ is said to be
  distributed according to an \emph{observation-driven model} with parameter
  $\theta\in\Theta$ if there is a bivariate Markov chain
  $\{(X_k,Y_k)\,:\,k\in\zsetp\}$ on $\Xset\times\Yset$ whose transition kernel
  $K^\theta$ satisfies
\begin{equation}\label{eq:def:observation-driven}
K^\theta((x,y);\rmd x'\rmd y') =  \delta_{\psi^\theta_{y}(x)}(\rmd x') \;
G^\theta (x';\rmd y') \eqsp,
\end{equation}
where $\delta_a$ denotes the Dirac mass at point $a$. Moreover, we will say
that the observation-driven time series model is dominated by  some $\sigma$-finite measure $\nu$ on
$(\Yset,\Ysigma)$ if for all $x \in \Xset$, the probability kernel $G^\theta (x;\cdot)$ is dominated by
$\nu$. In this case we denote by $g^\theta (x;\cdot)$ its Radon-Nikodym
derivative,
$g^\theta (x;y)=\frac{\rmd G^\theta (x;\cdot)}{\rmd
  \nu}(y)$, and we always assume that for all $(x,y) \in \Xset \times \Yset$ and for all $\theta\in\Theta$,
\begin{equation*}
g^\theta(x;y)>0\;.
\end{equation*}
\end{definition}
A dominated parametric observation-driven model is thus characterized by the
collection $\{(g^\theta,\psi^\theta)\,:\,\theta\in\Theta\}$.  The class of
observation-driven time series models is a particular case of
\emph{partially-observed Markov chains} since only $Y_k$'s are observed,
whereas $X_k$'s are \emph{hidden} variables.  Note that our notation for
observation-driven models is slightly different from that
of~\cite{dou:kou:mou:2013} where their sequence $\{Y_k\}$ corresponds to our
sequence $\{Y_{k-1}\}$.  Note also that the process $\{X_k\,:\,k\geq1\}$ by
itself is a Markov chain with transition kernel defined by
\begin{equation}
  \label{eq:Rtheta-def}
R^\theta(x;A )= \int \1_A(\psi^\theta_{y}(x))\; G^\theta(x;\rmd y),\quad x\in\Xset,\;A\in\Xsigma  \; .
\end{equation}
However, observation-driven time series models do not belong to the class of hidden
Markov models. This can be seen in the following recursive relation, which
holds for all $k\geq0$,
\begin{align}
\begin{split}
& X_{k+1}=\psi^\theta_{Y_{k}}(X_k) \eqsp, \label{eq:def:XY:theta:cl} \\
& Y_{k+1}\mid\mathcal{F}_k\sim G^\theta (X_{k+1};\cdot) \eqsp,
\end{split}
\end{align}
where $\mathcal{F}_k=\sigma\left(X_\ell, X_{\ell+1},Y_\ell\,:\,\ell\le k,
  \ell\in\zsetp\right)$ and which can be represented graphically as below.
\begin{figure}[h]
\begin{center}
 \begin{tikzpicture}[node  distance=1.5cm]
\node(a){${X_k}$};
\node(b)[right of =a]{${X_{k+1}}$};
\node(c)[below of =b]{${Y_{k+1}}$};
\node(d)[left of =c]{${Y_{k}}$};
\node(e)[right of =b]{${X_{k+2}}$};
\node(f)[left of =a]{};
\node(h)[right of =e]{};
\node(g)[right of =c]{${Y_{k+2}}$};

\draw[->, thick, color=blue!50!black](a)--(b)node[pos=0.7, below]{$\psi^\theta$};
\draw[->, thick](b)--(c)node[pos=0.4, right]{\textbf{$G^\theta$}};
\draw[->, thick, color=blue!50!black](d)--(b);
\draw[->, dashed, color=blue!50!black](c)--(e);
\draw[->, dashed](e)--(g);

\draw[->, dashed, color=blue!50!black](b)--(e);
\draw[->, dashed](a)--(d);
\draw[->, dashed,color=blue!50!black](f)--(a);
\draw[->, dashed, color=blue!50!black](e)--(h);

\end{tikzpicture}
\caption{Graphical representation of the observation-driven model.}\label{fig:or-arrow-graph}
\end{center}
\end{figure}
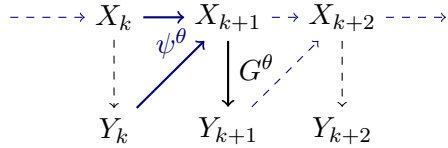

The most popular example is the GARCH(1,1) process, where
$G^\theta(x;\cdot)$ is a centered (say Gaussian) distribution with variance $x$ and
$\psi^\theta_y(x)$ is an affine function of $x$ and $y^2$. One can readily
check that Examples~\ref{example:Nbigarch:defi} and~\ref{example:nmgarch:defi}
are other instances of dominated observation-driven models.

The inference about model parameter is carried out by relying on the conditional
likelihood of the observations $(Y_1,\ldots,Y_n)$ given $X_1=x$ for an arbitrary
$x\in\Xset$. The corresponding conditional density function with respect to $\nu^{\otimes
  n}$ is, under parameter $\theta$, for all $x\in\Xset$,
  
\begin{equation} \label{eq:lkd:Y:cl:X1=x}
\chunk{y}{1}{n} \mapsto \prod_{k=1}^{n} g^\theta\left(\f{\chunk{y}{1}{k-1}}(x);y_{k}\right) \eqsp,
\end{equation}
where, for any vector $\chunk{y}{1}{p}=(y_1,\dots,y_p)\in\Yset^p$,
$\f{\chunk{y}{1}{p}}$ is the $\Xset\to\Xset$ function obtained as the successive
composition of  $\psi^\theta_{y_1}$,  $\psi^\theta_{y_2}$, ..., and $\psi^\theta_{y_p}$,
\begin{equation}
\label{eq:notationItere:f:cl}
\f{\chunk{y}{1}{p}}=\psi^\theta_{y_p} \circ \psi^\theta_{y_{p-1}} \circ \dots \circ \psi^\theta_{y_1}\,,
\end{equation}
with the convention $\f{\chunk{y}{s}{t}}(x)=x$ for $s>t$. Then, the corresponding
(conditional) Maximum Likelihood Estimator (MLE) $\mlY{x,n}$ of the parameter $\theta$,
 is defined by
\begin{equation}
\label{eq:defi:mle}
\mlY{x,n} \in \argmax_{\theta \in \Theta} \lkdM[x,n]{\chunk{Y}{1}{n}}[\theta]\eqsp,
\end{equation}
where
\begin{equation}
\label{eq:defi:lkdM}
\lkdM[x,n]{\chunk{y}{1}{n}}[\theta]\eqdef   n^{-1} \sum_{k=1}^{n} \ln g^\theta\left(\f{\chunk{y}{1}{k-1}} (x);y_k\right) \eqsp.
\end{equation}
In this contribution, we study the convergence of $\mlY{x,n}$ as $n\to\infty$
for some well-chosen value of $x$ under the assumption that the model is well specified and
the observations are in a steady state. This means that we assume that the
observations $\{Y_k \,:\,k\in\zsetp\}$ are distributed according to
$\tilde\PP^{\thv}$ with $\thv\in\Theta$, where, for all $\theta\in\Theta$,
$\tilde\PP^{\theta}$ denotes the stationary distribution of the observation-driven time series corresponding to the parameter $\theta$. However whether
such a distribution is well defined is not always obvious. We will use the
following ergodicity assumption.

\begin{hyp}{A}
\item\label{assum:gen:identif:unique:pi} For all $\theta\in\Theta$, the
  transition kernel $K^\theta$ of the complete chain admits a unique stationary
  distribution $\pi^\theta$ on $\Xset\times\Yset$.
\end{hyp}
With this assumption, we can now define $\tilde\PP^{\theta}$. The following
notation and
definitions will be used throughout the paper.
\begin{definition}\label{def:equi:theta}
For any probability distribution $\mu$ on
$\Xset\times\Yset$, we denote by $\PP_\mu^\theta$ the
distribution of the Markov chain $\{(X_k,Y_k),\,k\geq0\}$ with kernel
$K^\theta$ and initial probability mesure $\mu$.
Under Assumption~\ref{assum:gen:identif:unique:pi}, we denote by $\pi_1^\theta$
and $\pi_2^\theta$
the marginal distributions of $\pi^\theta$ on $\Xset$ and $\Yset$, respectively
and by $\PP^\theta$ and
$\tilde{\mathbb{P}}^\theta$ the probability distributions defined respectively as follows.
\begin{enumerate}[label=\alph*)]
\item $\PP^\theta$ denotes the extension of $\mathbb{P}_{\pi^\theta}^\theta$ on the whole line
$(\Xset\times\Yset)^{\zset}$.
\item $\tilde{\mathbb{P}}^\theta$ is the corresponding projection on the component $\Yset^{\zset}$.
\end{enumerate}
\end{definition}
The probability
distributions $\PP^\theta$ and $\tilde{\mathbb{P}}^\theta$ are more formally
defined by setting, for all $m\in\zset$ and
$B\in\Ysigma^{\otimes(m+\zsetpnz)}$,

\begin{equation}
  \label{eq:tildePtheta-def}
  \tilde{\mathbb{P}}^\theta\left(\Yset^{m+\zsetn}\times B\right)
=\mathbb{P}^\theta\left(\Xset^{\zset}\times \left(\Yset^{m+\zsetn}\times B\right)\right)=\mathbb{P}_{\pi^\theta}^\theta\left(\Xset^{m+\zsetpnz}\times B\right) \;,
\end{equation}
or equivalently, using the canonical functions $Y_k$, $k\in\zset$,
\begin{equation}   \label{eq:tildePtheta-def:2}
\tilde{\mathbb{P}}^\theta\left(\chunk{Y}{m+1}{\infty}\in B\right)
=\mathbb{P}^\theta\left(\chunk{Y}{m+1}{\infty}\in B\right)
=\mathbb{P}_{\pi^\theta}^\theta\left(\chunk{Y}{m+1}{\infty}\in B \right) \;.
\end{equation}
Here and in what follows, we abusively use the same notation $Y_k$
both for the canonical projection defined on $\Yset^{\zset}$ and for the
one defined on $(\Xset\times\Yset)^{\zsetp}$. We also use the symbols
$\mathbb{E}^\theta$ and $\tilde{\mathbb{E}}^\theta$ to denote the expectations
corresponding to $\PP^\theta$ and $\tilde{\mathbb{P}}^\theta$, respectively.

\section{Main results}\label{sec:main-result}

\subsection{Preliminaries}
In this section, we follow the same lines as in \cite{dou:kou:mou:2013} to
derive the convergence of the MLE $\mlY{x,n}$ for a general class of
observation-driven models. The approach is to establish that, as the number of
observations $n\to\infty$, there exists a
$(\Yset^\zset,\Ysigma^{\otimes\zset})\to(\rset,\mathcal{B}(\rset))$ measurable
function $p^{\theta}(\cdot|\cdot)$ such that the
normalized log-likelihood $\lkdM[x,n]{\chunk{Y}{1}{n}}[\theta]$ defined in \eqref{eq:defi:lkdM}, for some
appropriate value of $x$, can be approximated by
$$
n^{-1} \sum_{k=1}^n \ln p^{\theta}(Y_k|Y_{-\infty:k-1}) \;.
$$
To define $p^{\theta}(\cdot|\cdot)$, we set, for all
$\chunk{y}{-\infty}{1}\in\Yset^{\zsetn}$, whenever the following limit is well
defined,
  \begin{equation}
    \label{eq:def-p-theta-neq-thv}
p^{\theta}\left(y_1\,|\,\chunk{y}{-\infty}{0}\right) =
\begin{cases}
\displaystyle\lim_{m\to\infty}  g^\theta\left(\f[\theta]{\chunk{y}{-m}{0}}(x);y_1\right)&
\text{ if the limit exists,}\\
\infty&\text{ otherwise.}
\end{cases}
\end{equation}
By~\ref{assum:gen:identif:unique:pi}, the process $Y$ is ergodic under
$\tilde\PP^{\thv}$ and provided that
$$
\tilde\PE^\thv\left[\lnp p^{\theta}(Y_1|Y_{-\infty:0})\right]<\infty\eqsp,
$$
 it follows that
$$
\lim_{n\to\infty}\lkdM[x,n]{\chunk{Y}{1}{n}}[\theta]=\tilde\PE^\thv\left[\ln p^{\theta}(Y_1|Y_{-\infty:0})\right]\eqsp,\quad\tilde\PP^\thv\as
$$
In this paper we show that with probability tending to one, the MLE $\mlY{x,n}$
eventually lies in a neighborhood of the set
\begin{equation}\label{eq:def-Theta-star-set}
\Theta_\star=\argmax_{\theta\in\Theta}\tilde{\PE}^{\thv}\left[\ln
  p^{\theta}(Y_1|\chunk{Y}{-\infty}{0})\right] \; ,
\end{equation}
which only depends on $\thv$.  In this contribution, we provide easy-to-check sufficient
conditions implying
\begin{equation}\label{eq:strong-consistency-prelim}
\lim_{n\to\infty}\met(\mlY{x,n},\Theta_\star)=0,\quad \tilde\PP^{\theta_\star}\as,
\end{equation}
but, for the sake of brevity, we do not precisely determine the set
$\Theta_\star$. Many approaches have been proposed to investigate this problem,
which is often referred to as the \emph{identifiability} problem.  In
particular cases, one can prove that $\Theta_\star=\{\thv\}$, in which case the
strong consistency of the MLE follows
from~(\ref{eq:strong-consistency-prelim}). We will mention a general result
which precises how the set $\Theta_\star$ is related to the true parameter $\thv$ in~\autoref{rem:identifiability}. For the
moment, let us mention that we have
\begin{equation}\label{eq:basic-thv-Thetastar}
  \thv\in\Theta_\star\;,
\end{equation}
provided that the following assumption holds:
\begin{hyp}{B}
\item\label{assum:Ptheta:thetastar:is:density}
 For all $\theta,\thv\in\Theta$, we have
  \begin{enumerate}[label=(\roman*)]
  \item\label{item:PthetaNEQthetastar:is:density} If $\theta\neq\thv$,
    $y\mapsto p^{\theta}(y|\chunk{Y}{-\infty}{0})$ is a density function  $\tilde\PP^{\theta_\star}\as$
  \item\label{item:PthetaISthetastar:is:density}  Under
    $\tilde\PP^{\theta_\star}$, the function $y\mapsto
    p^{\thv}(y|\chunk{Y}{-\infty}{0})$ is the conditional density function
    of $Y_1$ given $\chunk{Y}{-\infty}{0}$.
  \end{enumerate}
\end{hyp}
Indeed, \eqref{eq:basic-thv-Thetastar} follows by  writing for all $\theta \in \Theta$,
\begin{align*}
&\tilde{\PE}^{\thv}\left[\ln
  p^{\thv}(Y_1|\chunk{Y}{-\infty}{0})-\ln
  p^{\theta}(Y_1|\chunk{Y}{-\infty}{0})\right]
  = \tilde{\PE}^{\thv}\left[\ln
  \frac{p^{\thv}(Y_1|\chunk{Y}{-\infty}{0})}{p^{\theta}(Y_1|\chunk{Y}{-\infty}{0})}\right] \\
& \quad
  = \tilde{\PE}^{\thv}\left[\tilde{\PE}^{\thv}\left[\left. \ln
  \frac{p^{\thv}(Y_1|\chunk{Y}{-\infty}{0})}
  {p^{\theta}(Y_1|\chunk{Y}{-\infty}{0})}\right| \chunk{Y}{-\infty}{0} \right] \right]\eqsp,
\end{align*}
which is nonnegative under \ref{assum:Ptheta:thetastar:is:density} since it is the expectation of a conditional Kullback-Leibler divergence.

\subsection{Convergence of the MLE}

In this part, we always assume that \ref{assum:gen:identif:unique:pi}
holds.
 The following is a list of
additional assumptions on which our convergence result relies.
 \begin{hyp}{A}
 \item\label{ass:21-lyapunov} There exists a
   function $\bar V:\Xset\to\rsetp$ such that, for all $\theta\in\Theta$, $\pi_1^\theta(
   \bar V)<\infty$.
 \end{hyp}
\begin{remark}
  Assumption~\ref{ass:21-lyapunov} is usually obtained as a byproduct of the
  proof of Assumption~\ref{assum:gen:identif:unique:pi}, see
  ~\autoref{sec:ergodicity}. It is here stated as an assumption for
  convenience.
\end{remark}
The following set of conditions can readily be checked on  $g^\theta$ and
$\psi^\theta$.
\begin{hyp}{B}
\item\label{assum:continuity:Y:g-theta}
For all $y \in\Yset$, the function $(\theta,x) \mapsto g^\theta(x;y)$ is
continuous on $\Theta\times\Xset$.
\item\label{assum:continuity:Y:phi-theta}
For all $y\in  \Yset$, the function $(\theta,x) \mapsto \psi^\theta_y(x)$ is
continuous  on $\Theta\times\Xset$.
\end{hyp}
The function $\bar V$ appearing
in~\ref{assum:practical-cond:O}\ref{item:22or23-barphi-V} below is the same one
as in Assumption~\ref{ass:21-lyapunov}. Moreover, in this condition and
throughout the paper we write  $f\lesssim V$ for a real-valued function $f$ and a nonnegative function $V$
defined on the same space $\Xset$, whenever there exists a positive constant
$c$ such that $|f(x)|\leq c V(x)$ for all $x\in\Xset$.
\begin{hyp}{B}
\item \label{assum:practical-cond:O}
There exist $x_1\in\Xset$, a closed set $\Xset_1\subseteq\Xset$,
  $\varrho\in(0,1)$, $C\geq0$ and
  measurable functions $\bar\psi:\Xset_1\to\rsetp$, $H:\rsetp\to\rsetp$ and
  $\bar\phi:\Yset\to\rsetp$ such that the following assertions hold.
\begin{enumerate}[label=(\roman*)]
\item \label{assum:exist:practical-cond:subX}
For all
  $\theta\in\Theta$ and $(x, y)\in\Xset \times \Yset$, $\psi^\theta_y(x)\in\Xset_1$.
\item\label{assum:momentCond} $\displaystyle\sup_{(\theta,x,y)\in\Theta\times\Xset_1\times\Yset}g^\theta(x;y)<\infty$.
\item  \label{assum:exist:practical-cond:O}
For all
  $\theta\in\Theta$, $n\in\zsetp$, $x\in\Xset$, and $\chunk{y}{1}{n}\in\Yset^n$,
  \begin{equation}
    \label{eq:exist:practical-cond:O}
\Xmet\left(\f{\chunk{y}{1}{n}}(x_1),\f{\chunk{y}{1}{n}}(x)\right)\leq \varrho^n \; \bar\psi(x)\;,
  \end{equation}

\item \label{item:psibar} $\bar\psi$ is locally bounded.
\item\label{item:barphi:bar:phi} For all $\theta\in\Theta$ and $y\in\Yset$, $\bar\psi(\psi_y^\theta(x_1))\leq\bar\phi(y)$.
\item\label{item:ineq:loggg} For all $\theta\in\Theta$ and $(x,x',y)\in\Xset_1\times\Xset_1\times\Yset$,
\begin{equation}\label{eq:ineq:loggg}
\left|\ln\frac{g^\theta(x;y)}{g^\theta (x';y)}\right|\leq H(\Xmet(x,x')) \;
\rme^{C \, \left(\Xmet(x_1,x)\vee \Xmet(x_1,x')\right)}\;\bar\phi(y)\;,
\end{equation}
\item \label{item:H} $H(u)=O(u)$ as $u\to0$.
\item \label{item:22or23-barphi-V} If $C=0$, then, for all $\theta\in\Theta$,
\begin{equation}
  \label{eq:barphi-V}
 G^\theta\lnp \bar{\phi} \lesssim \bar V \;,
\end{equation}
otherwise, for all $\theta\in\Theta$,
\begin{equation}
  \label{eq:barphi-strong-V}
  G^\theta\bar{\phi}\lesssim \bar V \;.
\end{equation}
\end{enumerate}
\end{hyp}
Let us now state our main result as follows.
\begin{theorem}\label{thm:convergence-main}
  Assume that~\ref{assum:gen:identif:unique:pi}, \ref{ass:21-lyapunov},
  \ref{assum:continuity:Y:g-theta}, \ref{assum:continuity:Y:phi-theta} and \ref{assum:practical-cond:O} hold.
  Then, letting $x_1\in\Xset$ as in~\ref{assum:practical-cond:O}, the function
  $p^{\theta}(\cdot|\cdot)$ defined by~\eqref{eq:def-p-theta-neq-thv} with $x=x_1$
  satisfies~\ref{assum:Ptheta:thetastar:is:density} and the
  convergence~(\ref{eq:strong-consistency-prelim}) of the MLE holds with the
  set $\Theta_\star$ defined by~\eqref{eq:def-Theta-star-set}.
\end{theorem}

For convenience, the proof is postponed to~\autoref{sec:convergence-mle}.

\begin{remark} \label{rem:misspecified}
As noticed in \cite{dou:kou:mou:2013}, the techniques used to prove
\autoref{thm:convergence-main} also apply in the misspecified case, where $Y$
is not distributed according to $\tilde\PP^\thv$. We do not pursue in this
direction in this contribution.
\end{remark}

The
consistency of the MLE then follows from \autoref{thm:convergence-main}  by the following remark.
\begin{remark}\label{rem:identifiability}
  In many specific cases, one can show that $\Theta_\star$ defined
  by~\eqref{eq:def-Theta-star-set} is the singleton $\{\thv\}$. However this
  task appears to be quite difficult in some cases such as
  \autoref{example:tingarch}. Instead one can use \cite[Section
  4.2]{douc2014maximizing}, where it is shown that the assumptions of
  \autoref{thm:convergence-main} 
    imply that $\Theta_\star$ is exactly the
  set of parameters $\theta$ such that
  $\tilde\PP^{\theta}=\tilde\PP^{\thv}$. Thus we can conclude that the MLE
  converges to the \emph{equivalence class} of the true parameter. This type of
  consistency has been introduced by \cite{leroux:1992} in the context of
  hidden Markov models in order to disentangle the proof of the consistency
  from the problem of identifiability. Recall that the model is identifiable if and
  only if the equivalent classes
  $\{\theta~:~\tilde\PP^{\theta}=\tilde\PP^{\thv}\}$ reduce to singletons
  $\{\thv\}$ for all $\thv\in\Theta$.
\end{remark}

\subsection{Ergodicity}
\label{sec:ergodicity}

In this section, the observation-driven model is studied to prove the
condition~\ref{assum:gen:identif:unique:pi}. Since this is a ``for all $\theta$
(...)''
condition, to save space and alleviate the notational burden, we will drop the
superscript $\theta$ from, for example, $G^\theta$, $R^\theta$ and
$\psi^\theta$ and respectively write $G$, $R$ and $\psi$, instead.

Ergodicity of Markov chains are usually studied using
$\psi$-irreducibility. This approach is well known to be quite efficient when
dealing with fully dominated models, see \cite{meyn:tweedie:2009}. It is not at
all the same picture for observation-driven models, where other tools need to
be invoked, see~\cite{fokianos:tjostheim:2011,dou:kou:mou:2013}. Since the ergodicity is
studied for a given parameter $\theta$, the ergodicity results
of~\cite{dou:kou:mou:2013} directly apply, even though
observation-driven models are restricted to the case where $g$ does not depend
on the unknown parameter $\theta$ in this reference. Our main contribution here
is to focus on an easy-to-check list of assumptions yielding the ergodicity
conditions~\ref{assum:gen:identif:unique:pi} and~\ref{ass:21-lyapunov}. We also
provide a lemma (\autoref{lem:det:alpha:gen}) which gives the construction of
the instrumental functions $\alpha$ and $\phi$ used in the list of
assumptions.

\begin{hyp}{A}
\item \label{ass-ergo-Xset} The measurable space $(\Xset, d)$ is a locally
  compact, complete and separable metric space and its associated $\sigma$-field
  $\Xsigma$ is the Borel $\sigma$-field.
\item \label{assum:weakFeller-V} There exist $(\lambda,\beta) \in (0,1) \times
  \rsetp$ and a measurable function $V: \Xset \to \rsetp$ such that $R V \leq
  \lambda V+\beta$ and $\{V\le M\}$ is compact for any $M>0$.
\item \label{assum:weakFeller}
The Markov kernel $R$ is weak Feller, that is, for any continuous and bounded
function $f$ defined on $\Xset$, $Rf$ is continuous and bounded on $\Xset$.
\item \label{assum:reachable}
The Markov kernel $R$ has  a reachable point, that is, there exists
$x_0\in\Xset$ such that, for any $x\in\Xset$ and any neighborhood $\mathcal{N}$ of $x_0$,
$R^m(x;\mathcal{N})>0$ for at least one positive integer $m$.
 \item\label{assum:bound:rho} We have $\displaystyle\sup_{\substack{(x,x',y)\in\Xset^2\times\Yset\\ x\neq x'}}\frac{\Xmet(\psi_y(x),\psi_y(x'))}{\Xmet(x,x')}<1$.
 \item \label{assum:38--40:alpha-phi} There exist a measurable function
   $\alpha$ from $\Xset^2$ to $[0,1]$, a measurable function
   $\phi:\Xset^2\to\Xset$ and a measurable function $W: \Xset^2 \to [1,\infty)$
   such that the following assertions hold.
  \begin{enumerate}[label=(\roman*)]
  \item\label{item:alpha:phi:3} For all $(x,x')\in\Xset^2$ and $y\in\Yset$,
\begin{equation}\label{eq:con:min-g-phi}
\min\left\{g(x;y),g(x';y)\right\}\ge\alpha(x,x')g\left(\phi(x,x');y\right)\eqsp.
\end{equation}

\item  \label{assum:definition-gamma-x} For all $x \in \Xset$, $W(x,\cdot)$ is finitely bounded in a
neighborhood of $x$, that is, there exists $\gamma_x>0$ such that
$\displaystyle\sup_{x' \in \ball{x}{\gamma_x}} W(x,x')<\infty$.

\item  \label{assum:hyp:beta:gene}
For all $(x,x')\in\Xset^2$,
$1-\alpha(x,x')\leq \Xmet(x,x') W(x,x')$.

\item  \label{assum:driftCond:W:new}
$\displaystyle
 \sup \left( \int_{\Yset}W(\psi_y(x),\psi_y(x'))\,G(\phi(x,x');\rmd y) - W(x,x')\right)<\infty$,
where the sup is taken over all $(x,x')\in\Xset^2$.
  \end{enumerate}
\end{hyp}

We can now state the main ergodicity result.

\begin{theorem}\label{thm:ergodicity}
  Conditions~\ref{ass-ergo-Xset}, \ref{assum:weakFeller-V}, \ref{assum:weakFeller}, \ref{assum:reachable}, \ref{assum:bound:rho}
  and~\ref{assum:38--40:alpha-phi} imply that $K$ admits a unique stationary
  distribution $\pi$ on $\Xset\times\Yset$. Moreover $\pi_1\bar V<\infty$  for
  every $\bar V:\Xset\to\rsetp$ such that $\bar V\lesssim V$.
\end{theorem}
The proof of \autoref{thm:ergodicity} is postponed to
\autoref{sec:ergodicity-proofs} for convenience.

The first conclusion of \autoref{thm:ergodicity} can directly be applied for
all $\theta\in\Theta$ to check~\ref{assum:gen:identif:unique:pi}. The second
conclusion can be used to check~\ref{ass:21-lyapunov}. In doing so,
one must take care of the fact that although $V$ may depend on $\theta$, $\bar
V$ does not.

Assumptions~\ref{assum:weakFeller-V},~\ref{assum:weakFeller} and~\ref{assum:reachable}
have to be checked directly on the Markov kernel $R$ defined
by~(\ref{eq:Rtheta-def}). To this end it can be useful to define, for any given $x\in\Xset$, the
distribution
  \begin{equation}
    \label{eq:def-Px}
\bar\PP_x := \PP_{\delta_x\otimes G(x;\cdot)}
  \end{equation}
on $(\Xset\times\Yset)^{\zsetp}$,  where $\PP_\mu$ is
  defined for any distribution $\mu$ on $\Xset\times\Yset$ as in \autoref{def:equi:theta}. Then the first component process
  $\{X_k,\,k\in\zsetp\}$ associated to $\bar\PP_x$ is a Markov chain with Markov kernel $R$ and initial
  distribution $\delta_x$.

We now provide a general
framework for constructing $\alpha$ and $\phi$ that appear in~\ref{assum:38--40:alpha-phi}.
\begin{lemma}\label{lem:det:alpha:gen}
  Suppose that $\Xset=\Cset^\Sset$ for some measurable space $(\Sset, \Ssigma)$
  and $\Cset\subseteq\rset$. Thus for all $x\in\Xset$, we write $x=(x_s)_{s\in
    \Sset}$, where $x_s\in\Cset$ for all $s\in\Sset$. Suppose moreover that for
  all $x=(x_s)_{s\in\Sset}\in\Xset$, we can express the conditional density
  $g(x;\cdot)$ as a mixture of densities of the form $j(x_s)h(x_s;\cdot)$ over
  $s\in\Sset$.  This means that for all $t\in\Cset$, $y\mapsto j(t)h(t;y)$ is a
  density with respect to $\nu$ and there exists a probability measure
  $\mu$ on $(\Sset, \Ssigma)$ such that
  \begin{equation}
    \label{eq:def-g-j-h-mu}
g(x;y)=\int_\Sset {j(x_s)h(x_s;y)\mu(\rmd s)}\;,\quad y\in\Yset\;.
\end{equation}
We moreover assume that $h$ takes non-negative values and
that one of the two following assumptions holds.
\begin{hyp}{F}
\item\label{item:F1} For all $y\in\Yset$, the function $h(\cdot;y): t\mapsto h(t;y)$ is non-decreasing.
\item\label{item:F2} For all $y\in\Yset$, the function $h(\cdot;y): t\mapsto h(t;y)$ is non-increasing.
\end{hyp}
For all $(x,x')\in\Xset^2$, denoting $x\wedge x':=(\min\{x_s,x_s'\})_{s\in \Sset}$ and $x\vee
x':=(\max\{x_s,x_s'\})_{s\in \Sset}$,  we define $\alpha(x,x')$ and
$\phi(x,x')$ as
$$
\begin{cases}
\displaystyle\alpha(x,x')=\inf_{s\in \Sset}{\left\{\frac{j(x_s\vee x'_s)}{j(x_s\wedge
      x'_s)}\right\}} \quad\text{and}\quad\phi(x,x')=x\wedge x'
&\text{ under\emph{~\ref{item:F1}}}\;;\\
\displaystyle\alpha(x,x')=\inf_{s\in \Sset}{\left\{\frac{j(x_s\wedge x'_s)}{j(x_s\vee x'_s)}\right\}}
\quad\text{and}\quad\phi(x,x')=x\vee x'
&\text{ under\emph{~\ref{item:F2}}}\;.
\end{cases}
$$
Then $\alpha$ and $\phi$ defined above
satisfy~\ref{assum:38--40:alpha-phi}\ref{item:alpha:phi:3}.
\end{lemma}
\begin{proof}
We only prove this result under Condition~\ref{item:F1}. The proof is similar
under~\ref{item:F2}.

Since for
all $t\in\Cset$, $y\mapsto j(t)h(t;y)$ is a density with respect to $\nu$, we have
\begin{equation*}
j(t)=\left(\int{h(t;y)\nu(\rmd y)}\right)^{-1}>0\;.
\end{equation*}
Thus $j$ is non-increasing on $\Cset$.
Clearly, the defined $\alpha$ takes values on $[0,1]$ and $\phi$
defines a function from $\Xset^2$ to $\Xset$.  For all $(x, x')\in\Xset^2$ and $y\in\Yset$, we have
\begin{align*}
g(x;y)&=\int_\Sset{j(x_s)h(x_s;y)\mu(\rmd s)} \\
        &\ge \int_\Sset{j(x_s\vee x_s')h(x_s\wedge x_s';y)\mu(\rmd s)}\\
        &\ge \int_\Sset{\frac{j(x_s\vee x_s')}{j(x_s\wedge x_s')}j(x_s\wedge x_s')h(x_s\wedge x_s';y)\mu(\rmd s)}\\
        &\ge \int_\Sset{\inf_{s\in \Sset}\left\{\frac{j(x_s\vee x_s')}{j(x_s\wedge x_s')}\right\}j(x_s\wedge x_s')h(x_s\wedge x_s';y)\mu(\rmd s)}\\
        &=\alpha(x,x')g(\phi(x,x');y)\eqsp.
\end{align*}
By symmetry of $\alpha$ and $\phi$, we get~\eqref{eq:con:min-g-phi} and thus
\ref{assum:38--40:alpha-phi}\ref{item:alpha:phi:3} holds. 
\end{proof}

\section{Examples}\label{sec:examples}
Let us now apply these results to prove the convergence of MLE of Examples~\ref{example:Nbigarch:defi}, ~\ref{example:nmgarch:defi} and ~\ref{example:tingarch}.


\subsection{NBIN-GARCH model}
\label{sec:nbin-garch-model}
\autoref{example:Nbigarch:defi} is a specific case of \autoref{def:obs-driv:gen} where $\nu$ is the counting measure on $\Yset=\nset$,
\begin{align}
&\psi^\theta_y(x)=\omega+ax+by\eqsp,\label{eq:Nbingarch:def:phi}\\
&g^\theta(x;y)
=\frac{\Gamma(y+r)}{y!\Gamma(r)}\left(\frac{1}{1+x}\right)^r\left(\frac{x}{1+x}\right)^y \eqsp,\label{eq:Nbingarch:g}
\end{align}
with $\theta=(\omega, a, b, r)$ in a compact subset $\Theta$ of
$(0,\infty)^4$ and $\Xset=(0,\infty)$.


In \cite[Theorem~1]{zhu:2011},
the equation satisfied by the mean of the observations $\mu_k=\PE[Y_k]$ is
derived and is shown to admit a constant solution if and only if
\begin{equation}
  \label{eq:nbingarch-stability-cond-zhu}
  rb+a<1 \;.
\end{equation}
This clearly implies that this condition is necessary to have a stationary
solution $\{Y_k\}$ with finite mean. However it does not imply the existence of
such a solution. In fact, the following result shows
that~(\ref{eq:nbingarch-stability-cond-zhu}) is indeed a necessary and
sufficient condition to have a stationary solution $\{Y_k\}$ with finite
mean. It also shows that all the assumptions of \autoref{thm:convergence-main}
hold, which, with \autoref{rem:identifiability}, provides the consistency of the
MLE $\mlY{x_1,n}$ for any  $x_1\in\Xset$.

\begin{theorem}\label{theo:ergo-converge:nbigarch}
  Suppose that all $\theta=(\omega, a, b, r)$ in $\Theta$ satisfy
  Condition~\eqref{eq:nbingarch-stability-cond-zhu}. Then
  Assumptions~\ref{assum:gen:identif:unique:pi}, \ref{ass:21-lyapunov},
  \ref{assum:continuity:Y:g-theta}, \ref{assum:continuity:Y:phi-theta} and \ref{assum:practical-cond:O} hold
  with $\bar V$ being defined as the identity function on $\Xset$ and with any $x_1\in\Xset$.
\end{theorem}



\begin{proof} 
For convenience, we divide the proof into two
  steps.

\noindent\textbf{Step 1.} We first prove
Assumptions~\ref{assum:gen:identif:unique:pi} and \ref{ass:21-lyapunov} by
applying \autoref{thm:ergodicity}. We set $\bar V(x)= V(x)=x$ and thus we only need to
check~~\ref{ass-ergo-Xset}, \ref{assum:weakFeller-V}, \ref{assum:weakFeller}, \ref{assum:reachable},
\ref{assum:bound:rho} and~\ref{assum:38--40:alpha-phi}. Condition~\ref{ass-ergo-Xset} holds. We have for all $\theta\in\Theta$,
$$
R V(x)=\omega + (a+br)x =(a+br)V(x) + \omega,
$$
which yields  \ref{assum:weakFeller-V}.
The fact that the kernel $R$ is weak Feller easily
follows by observing that, as $p\to p'$, $\mathcal{NB}(r,p)$ converges weakly
to $\mathcal{NB}(r,p')$, so \ref{assum:weakFeller} holds.

We now prove \ref{assum:reachable}. Let $x_\infty=\omega/(1-a)$. Let $x \in
\rset$ and define recursively the sequence $x_0=x, x_k=\omega+ax_{k-1}$ for all
positive integers $k$. Since $0<a<1$, this sequence converges to the fixed
point $x_\infty$. Therefore, defining $\PP_x$ as in~(\ref{eq:def-Px}), for any
neighborhood $\mathcal{N}$ of $x_\infty$, there exists some $n$ such that
$x_n\in \mathcal{N}$ and we have

\begin{align*}
R^{n}(x; \mathcal{N})=\bar\PP_{x}\left(X_n\in
  \mathcal{N}\right)&\ge\bar\PP_{x}\left(X_k=x_k\text{ for all $k=1,\dots,n$}\right)\\
&=\bar\PP_{x}\left(Y_0=\ldots=Y_{n-1}=0\right)>0.
\end{align*}
So \ref{assum:reachable} holds.
Assumption~\ref{assum:bound:rho} holds since we have for all
$(x,x',y)\in\Xset^2\times\Yset$ with $x\neq
x'$, $$\frac{|\psi_y(x)-\psi_y(x')|}{|x-x'|}=a<1\eqsp.
$$
To prove~\ref{assum:38--40:alpha-phi}, we apply ~\autoref{lem:det:alpha:gen}
with $\Cset=\Xset$, $\Sset=\{1\}$ (so $\mu$ boils down to the Dirac measure on $\{1\}$). For all
$(x,y)\in\Xset\times\Yset$, let $j(x)=\left(\frac{1}{1+x}\right)^r$ and
$h(x;y)=\frac{\Gamma(y+r)}{y!\Gamma(r)}\left(\frac{x}{1+x}\right)^y$. Indeed,
$h$ satisfies ~\ref{item:F1}. Thus
by~\autoref{lem:det:alpha:gen}, for all $(x,x')\in\Xset^2$ and $y\in\Yset$, we
get that
\begin{equation*}
\alpha(x, x')=\left(\frac{1+x \wedge x'}{1+x \vee x'}\right)^r\in(0,1] \quad \text{and} \quad \phi(x,x')=x\wedge x'
\end{equation*}
satisfy~\ref{assum:38--40:alpha-phi}\ref{item:alpha:phi:3}.
For any given $r>0$,  let a function $W\,:\,\Xset^2\to[1,\infty)$  be defined by, for all $(x,x')\in\Xset^2,\, W(x,x')=1\vee r$. By definition of $W$, as a constant function, \ref{assum:38--40:alpha-phi}\ref{assum:definition-gamma-x} and \ref{assum:38--40:alpha-phi}\ref{assum:driftCond:W:new}  clearly hold. Moreover, \ref{assum:38--40:alpha-phi}\ref{assum:hyp:beta:gene} holds since
for all $(x, x') \in\Xset^2$, we have that
\begin{equation*}
1-\alpha(x,x')\leq (1\vee r)\lvert x-x'\rvert = W(x,x')\lvert x-x'\rvert\eqsp.
\end{equation*}
Therefore, \ref{assum:38--40:alpha-phi} holds, which completes \textbf{Step 1}.

\noindent\textbf{Step 2}. We now prove  
  \ref{assum:continuity:Y:g-theta}, \ref{assum:continuity:Y:phi-theta} and \ref{assum:practical-cond:O}.
By assumption on $\Theta$, then there exists $(\underline{\omega},\bar{\omega},\underline{b}, \bar b,\underline{r}, \bar r,\underline{\alpha}, \bar \alpha) \in (0,\infty)^6\times(0,1)^2$ such that
$$
\underline{\omega}\leq \omega \leq\bar \omega, \,\, \leq b \leq\bar{b}, \,\, \underline{r}\leq r\leq \bar{r}, \,\, \underline{\alpha}\leq a+br\leq \bar{\alpha}\eqsp.
$$
Clearly, \ref{assum:continuity:Y:g-theta} and
\ref{assum:continuity:Y:phi-theta} hold by definitions of $\psi^\theta_y(x)$
and $g^\theta(x;y)$.  It remains to check~\ref{assum:practical-cond:O} for a
well-chosen closed subset $\Xset_1$ and any $x_1\in\Xset$. Let
$\Xset_1=[\underline \omega,\infty)\subset\Xset$ so that
\ref{assum:practical-cond:O}\ref{assum:exist:practical-cond:subX} holds. By
noting that for all $(\theta,x,y)\in\Theta\times\Xset\times\Yset$,
$g^\theta(x;y)\le1$, we have
\ref{assum:practical-cond:O}\ref{assum:momentCond}.
From~\eqref{eq:notationItere:f:cl} and ~\eqref{eq:Nbingarch:def:phi}, we have
for all $s\leq t$, $y_{s:t}\in\Yset^{t-s+1}$, $x\in\Xset$ and
$\theta\in\Theta$,
\begin{equation}\label{eq:recurs:nbingarch}
\f{\chunk{y}{s}{t}}(x)= \omega\left(\frac{1-a^{t-s+1}}{1-a}\right)+a^{t-s+1}x+b\sum_{j=0}^{t-s} a^jy_{t-j} \eqsp.
\end{equation}
Using~\eqref{eq:recurs:nbingarch}, we have, for all $\theta\in\Theta$, $x\in\Xset$ and $y_{1:n}\in\Yset^n$,
\begin{equation*}
\left|\f{\chunk{y}1n}(x_1)-\f{\chunk{y}1n}(x)\right|= a^n\left|x_1-x\right|\le\bar{\alpha}^n\left|x_1-x\right|\eqsp.
  \end{equation*}
  This gives~\ref{assum:practical-cond:O}\ref{assum:exist:practical-cond:O}
  and~\ref{assum:practical-cond:O}\ref{item:psibar} by setting
  $\varrho=\bar{\alpha}<1$ and $\bar\psi(x)=|x_1-x|$.  Next we set
  $\bar\phi$, $H$ and $C$
  to meet Conditions~\ref{assum:practical-cond:O}\ref{item:barphi:bar:phi}
  and~\ref{assum:practical-cond:O}\ref{item:ineq:loggg} and~\ref{assum:practical-cond:O}\ref{item:H}. Let us write, for all
  $\theta\in\Theta$ and $y\in\Yset$,
$$
\left|x_1-\psi^\theta_y(x_1)\right|\le\omega+(1+a)x_1+by\le\bar\omega+(1+\bar{\alpha})x_1+\bar{b}y
$$
and, for all $(x,x')\in\Xset_1^2=[\underline \omega,\infty)^2$,
\begin{align*}
\left|\ln \frac{g^\theta (x;y)}{g^\theta (x';y)}\right|
&=\left|(r+y)\left[\ln(1+x')-\ln(1+x)\right]+y\left[\ln x-\ln x'\right]\right|\\
&\leq \left[(r+y)(1+\underline \omega)^{-1}+y\,\underline
  \omega^{-1}\right]\;|x-x'|\\
&\leq \left[\overline r +y\,(1+\underline
  \omega^{-1})\right]\;|x-x'| \;.
\end{align*}
Setting  $\bar\phi(y)=\bar\omega\vee{\bar r}+(1+\bar{\alpha})x_1+\left(\bar{b}\vee(1+\underline \omega^{-1})\right)y$,
$H(x)=x$ and $C=0$ then yield Conditions~\ref{assum:practical-cond:O}\ref{item:barphi:bar:phi}, ~\ref{assum:practical-cond:O}\ref{item:ineq:loggg}
  and~\ref{assum:practical-cond:O}\ref{item:H}. Now
  \ref{assum:practical-cond:O}\ref{item:22or23-barphi-V} follows from
$$
\int \lnp y \; G^\theta(x,\rmd y) \leq \int y \; G^\theta(x,\rmd y) = r x\leq
\overline r \bar V(x)\;.
$$
This concludes the proof.
\end{proof}


\subsection{NM-GARCH model}

The NM$(d)$-GARCH$(1,1)$ of~\autoref{example:nmgarch:defi} is a specific case of \autoref{def:obs-driv:gen} where $\Xset=\rsetp^d$ and $\nu$ is the Lebesgue measure on $\Yset=\rset$,
\begin{align}
&\psi^\theta_y(\mathbf{x})=\boldsymbol{\omega}+\mathbf{A}\mathbf{x}+y^2\mathbf{b} \eqsp,
\label{eq:psi:nm-garch}\\
&g^\theta(\mathbf{x};y)= \sum_{\ell=1}^{d}{\gamma_\ell\frac{\rme^{-y^2/{2x_{\ell}}}}{(2\pi x_{\ell})^{1/2}}}\eqsp, \quad (\mathbf{x},y)\in\Xset\times\Yset\eqsp,
\label{eq:g-theta:nm-garch}
\end{align}
and $\theta=\left(\boldsymbol{\gamma}, \boldsymbol{\omega}, \mathbf{A},
  \mathbf{b}\right)\in\Theta$, a compact subset of
$\simplex_d\times(0,\infty)^d\times\rsetp^{d\times d}\times\rsetp^d$, with
$\simplex_d$ defined by~(\ref{eq:simplex-def}).

In \cite{haas2004mixed}, it is shown that the equation satisfied by the
variance of a univariate NM($d$)-GARCH$(1,1)$ process admits a constant solution if and only if
\begin{equation} \label{eq:con:station:nmgarch}
\radius(\mathbf{A}+\mathbf{b}\boldsymbol{\gamma}^T)<1\eqsp,
\end{equation}
where, for any square matrix $\mathbf{M}$, $\radius(\mathbf{M})$ denotes the
spectral radius of $\mathbf{M}$. It follows that the existence of a weakly
stationary solution implies \eqref{eq:con:station:nmgarch} but it does not say
anything about the existence of stationary or weakly stationary solution. The
result below shows that \eqref{eq:con:station:nmgarch} is indeed a sufficient
condition for the existence of a stationary solution with finite variance. It
moreover provides with \autoref{thm:convergence-main} and
\autoref{rem:identifiability} the consistency of the MLE $\mlY{\mathbf{x}_1,n}$
for any $\mathbf{x}_1\in\Xset$.


 \begin{theorem}\label{theo:ergo-convergence:nmgarch}
  Suppose that all $\theta=\left(\boldsymbol{\gamma}, \boldsymbol{\omega}, \mathbf{A}, \mathbf{b}\right)$ in $\Theta$ satisfy
  Condition~\eqref{eq:con:station:nmgarch}. Then
  Assumptions~\ref{assum:gen:identif:unique:pi}, \ref{ass:21-lyapunov},
  \ref{assum:continuity:Y:g-theta}, \ref{assum:continuity:Y:phi-theta} and \ref{assum:practical-cond:O} hold
  with $\bar V$ being defined as any norm on $\Xset$.
\end{theorem}


\begin{proof} In this proof section, we set
  \begin{equation}
    \label{eq:Vbar-nmgarch}
\bar V(\mathbf{x})=|\mathbf{x}|= \sum_{\ell=1}^d |x_\ell|\;,
  \end{equation}
for all $\mathbf{x}=(x_\ell)\in\Xset$. As
in~\autoref{theo:ergo-converge:nbigarch}, we divide the proof into two
steps.

\noindent\textbf{Step 1.} We first show that
Assumptions~\ref{assum:gen:identif:unique:pi} and \ref{ass:21-lyapunov} hold
with the above $\bar V$ by applying \autoref{thm:ergodicity}. Define $V$ on
$\Xset$ by setting 
$$
V(\mathbf{x})= (\mathbf{1}+\mathbf{x}_0)^T\mathbf{x} \;,
$$
where $\mathbf{1}$ is the vector of $\Xset$ with all entries equal to 1 and
$\mathbf{x}_0$ is defined by
$$
\mathbf{1}+\mathbf{x}_0 =
(\mathbf{I}-\left(\mathbf{A}+\mathbf{b}\boldsymbol{\gamma}^T\right)^T)^{-1}\mathbf{1} \;.
$$
We indeed note that by Condition~(\ref{eq:con:station:nmgarch}) the above
inversion is well defined and moreover
$$
(\mathbf{I}-(\mathbf{A}+\mathbf{b}\boldsymbol{\gamma}^T)^T)^{-1} = \mathbf{I}+ \sum_{k\geq1}
\left(\mathbf{A}^T+\boldsymbol{\gamma}\mathbf{b}^T\right)^k \;,
$$
and, since $\mathbf{A}$, $\mathbf{b}$, $\boldsymbol{\gamma}$ all have
non-negative entries, it follows that $\mathbf{x}_0$ has non-negative
entries. Thus, for all $\mathbf{x}=(x_\ell)\in\Xset$,
$$
\bar V(\mathbf{x}) = \mathbf{1}^T\mathbf{x} \leq V(\mathbf{x}) \;,
$$
so that $\bar V\lesssim V$. Hence by \autoref{thm:ergodicity}, we thus only
need to check~\ref{ass-ergo-Xset}, \ref{assum:weakFeller-V},
\ref{assum:weakFeller}, \ref{assum:reachable}, \ref{assum:bound:rho}
and~\ref{assum:38--40:alpha-phi} with $V$ defined as above for a given
$\theta=\left(\boldsymbol{\gamma}, \boldsymbol{\omega}, \mathbf{A},
  \mathbf{b}\right)\in\Theta$ (so we drop $\theta$ in the notation in the
remaining of \textbf{Step 1}). Condition~\ref{ass-ergo-Xset} holds for any
metric $\Xmet$ associated to a norm on the finite dimensional space $\Xset$. (The precise choice of $d$ is
postponed to the verification of~\ref{assum:bound:rho}.) We have
\begin{align*}
RV(\mathbf{x})
&=\int V(\boldsymbol{\omega}+\mathbf{A}\mathbf{x}+y^2\mathbf{b})\;G(\mathbf{x},\rmd y)\\
&=
(\mathbf{1}+\mathbf{x}_0)^T\boldsymbol{\omega}+(\mathbf{1}+\mathbf{x}_0)^T\left(\mathbf{A}+\mathbf{b}\boldsymbol{\gamma}^T\right)\mathbf{x}\\
&=V(\boldsymbol{\omega})+\mathbf{1}^T(\mathbf{I}-\left(\mathbf{A}+\mathbf{b}\boldsymbol{\gamma}^T\right))^{-1}\left(\mathbf{A}+\mathbf{b}\boldsymbol{\gamma}^T-\mathbf{I}+\mathbf{I}\right)\mathbf{x}\\
&=V(\boldsymbol{\omega})+\mathbf{x}_0^T\mathbf{x} \\
&\leq V(\boldsymbol{\omega})+\lambda V(\mathbf{x})\;,
\end{align*}
where we set
$\lambda=\max_\ell\left\{\mathbf{x}_{0,\ell}/(1+\mathbf{x}_{0,\ell})\right\}<1$. Hence~\ref{assum:weakFeller-V}
holds. Condition \ref{assum:weakFeller}
easily follows from the continuity of the Gaussian distribution with respect to
its variance parameter. We now prove \ref{assum:reachable}.
From~\eqref{eq:notationItere:f:cl} and \eqref{eq:psi:nm-garch},
we have for all $n\geq1$, $y_{0:n-1}\in\Yset^{n}$ and $\mathbf{x}\in\Xset$,
\begin{equation}\label{eq:psi:nm-garch:iterate}
\f{\chunk{y}{0}{n-1}}(\mathbf{x})=
\mathbf{A}^{n}\mathbf{x}+\sum_{j=0}^{n-1}
\mathbf{A}^j (\boldsymbol{\omega}+y_{n-1-j}^2\mathbf{b})\eqsp.
\end{equation}
Let us use the norm
$$
\|\mathbf{M}\|= \max_j \sum_i |\mathbf{M}_{i,j}|=\sup_{|\mathbf{x}|\leq1}|\mathbf{M}\mathbf{x}|
$$ 
on $d\times d$ matrices. Note that
by~(\ref{eq:con:station:nmgarch}), there exists $\delta\in(0,1)$ and $c>0$ such
that, for any $k\geq1$,
\begin{equation}
  \label{eq:norm-control}
\left\|\left(\mathbf{A}+\mathbf{b}\boldsymbol{\gamma}^T\right)^k\right\|\leq c\;
\delta^k \;.
\end{equation}
Using that  $\mathbf{A}$, $\mathbf{b}$, $\boldsymbol{\gamma}$ all have
nonnegative entries, we have
\begin{equation}
  \label{eq:norm-control2}
\left\|\mathbf{A}^k\right\|\leq\left\|\left(\mathbf{A}+\mathbf{b}\boldsymbol{\gamma}^T\right)^k\right\|\;.
\end{equation}
Hence $(\mathbf{I}-\mathbf{A})^{-1}=\mathbf{I}+\sum_{k\ge1}\mathbf{A}^k$ is well defined and we set
$\mathbf{x}_\infty=\mathbf{(I-A)}^{-1}\boldsymbol{\omega}$ so that,
with~(\ref{eq:psi:nm-garch}), we have
$$
\f{\chunk{y}{0}{n-1}}(\mathbf{x})-\mathbf{x}_\infty=
\mathbf{A}^{n}\mathbf{x}+ \sum_{j\geq n}\mathbf{A}^j \boldsymbol{\omega} +
\sum_{j=0}^{n-1}
y_{n-1-j}^2\mathbf{A}^j \mathbf{b}\;.
$$
Then, using definition~(\ref{eq:def-Px}), we get that, $\bar\PP_{\mathbf{x}}$\as, for all $n\geq1$, 
\begin{align*}
  |\mathbf{X}_n-\mathbf{x}_\infty|&=\left|\f[]{\chunk{Y}0{n-1}}(\mathbf{x})-\mathbf{x}_\infty\right|\\
&\leq
\left|\mathbf{A}^n(\mathbf{x}-\mathbf{x}_\infty)\right|+ \sum_{j\geq n}\left|\mathbf{A}^j \boldsymbol{\omega}\right|+
\left(\max_{0\leq j\leq n-1}Y_j^2\right)\,\sum_{j=0}^{n-1}\left|\mathbf{A}^j\mathbf{b}\right|\;.
\end{align*}
With~(\ref{eq:norm-control}) and~(\ref{eq:norm-control2}), this implies
$$
\bar\PP_{\mathbf{x}}\left(|\mathbf{X}_n-\mathbf{x}_\infty|\leq c \left[\delta^n
    \left(|\mathbf{x}-\mathbf{x}_\infty|+\frac{|\boldsymbol{\omega}|}{1-\delta}\right)+\frac{|\mathbf{b}|}{1-\delta}\,\max_{0\leq
      j\leq n-1}Y_j^2\right]\right) = 1  \;.
$$
To obtain \ref{assum:reachable}, it is sufficient to observe that, since $g$
takes positive values in~(\ref{eq:g-theta:nm-garch}), for any
positive $\epsilon$, $\mathbf{x}\in\Xset$ and any $n\geq1$,
$$
\bar\PP_{\mathbf{x}}\left(\max_{0\leq j\leq n-1}Y_j^2<\epsilon\right)>0 \;.
$$
Next we prove~\ref{assum:bound:rho}. We have
$$
\psi_y(\mathbf{x})-\psi_y(\mathbf{x'})=\mathbf{A}(\mathbf{x}-\mathbf{x'}) \;.
$$
Since~(\ref{eq:norm-control}) and~(\ref{eq:norm-control2}) imply that
$\radius(\mathbf{A})<1$, there exists a vector norm which makes $\mathbf{A}$
strictly contracting. Choosing the metric $\Xmet$ on $\Xset$ as the one
derived from this norm, we get~\ref{assum:bound:rho}. To
show~\ref{assum:38--40:alpha-phi}, we again rely
on~\autoref{lem:det:alpha:gen}. Let us set $\Cset=(0,\infty)$ and
$\Sset=\{1,\ldots, d\}$ and define the probability measure $\mu$ on $\Sset$ by
$\mu(\{s\})=\gamma_s$, for all $s\in\Sset$. For all $(t,y)\in\Cset\times\Yset$,
let $j(t)=\frac{1}{(2\pi t)^{1/2}}$ and
$h(t;y)=\exp\left(-y^2/{2t}\right)$. Obviously, 
Relation~(\ref{eq:def-g-j-h-mu}) holds and $h$ 
satisfies ~\ref{item:F1}. Hence, ~\autoref{lem:det:alpha:gen} implies that
$\alpha$ and $\phi$ defined respectively for all $\mathbf{x}=(x_1, \dots,
x_d), \ \mathbf{x'}=(x_1', \dots, x_d')\in\Xset$ by
\begin{equation*}
\alpha(\mathbf{x,x'})=\min_{1\le\ell\le d}{\left\{\left(\frac{x_\ell\wedge x'_\ell}{x_\ell\vee x'_\ell}\right)^\frac{1}{2}\right\}}\in(0,1] \  \text{and}  \ \phi\mathbf{(x,x')}=(x_1\wedge x'_1,\ldots,x_d\wedge x'_d),
\end{equation*}
satisfy ~\ref{assum:38--40:alpha-phi}\ref{item:alpha:phi:3}.  For $\mathbf{x}=(x_1,
\dots, x_d), \ \mathbf{x'}=(x_1', \dots, x_d')\in\Xset$, we have
\begin{align*}
1-\alpha(\mathbf{x,x'}) &=1-\min_{1\le\ell\le d}{\left\{\left(1-\frac{|x_\ell-x'_\ell|}{x_\ell\vee x'_\ell}\right)^\frac{1}{2}\right\}}\\
&\leq \max_{1\le\ell\le d}{\left\{\frac{|x_\ell-x'_\ell|}{x_\ell\vee
      x'_\ell}\right\}}\\
&\leq  \min_{1\le\ell\le d}(x_\ell^{-1}\wedge x_\ell^{\prime-1})\;|\mathbf{x}-\mathbf{x}'|\\
& \leq W(\mathbf{x},\mathbf{x}')\;\Xmet(\mathbf{x},\mathbf{x}') \;,
\end{align*}
where $\Xmet$ is the metric previously defined and $W$ is defined by
$W(\mathbf{x},\mathbf{x}')=1\vee \left(c_{\Xmet} \, \min_{1\le\ell\le
    d}(x_\ell^{-1}\wedge x_\ell^{\prime-1})\right)$ with $c_{\Xmet}>0$ is
conveniently chosen (such a constant exists since $\Xmet$ is the metric
associated to a norm and $\Xset$ has finite dimension). Then
\ref{assum:38--40:alpha-phi}\ref{assum:definition-gamma-x} and
\ref{assum:38--40:alpha-phi}\ref{assum:hyp:beta:gene} hold and, since for all
$y\in\Yset$ and $x\in\Xset$, $\psi_y(\mathbf{x})$ has all its entries bounded
from below by the positive entries of $\boldsymbol{\omega}$,
$W(\psi_y(\mathbf{x}),\psi_y(\mathbf{x}'))$ is uniformly bounded over
$(\mathbf{x, x'}, y)\in\Xset\times\Xset\times\Yset$
and~\ref{assum:38--40:alpha-phi}\ref{assum:driftCond:W:new} holds. This
completes \textbf{Step 1}.


\noindent\textbf{Step 2} We now show that Assumptions  
  \ref{assum:continuity:Y:g-theta}, \ref{assum:continuity:Y:phi-theta} and \ref{assum:practical-cond:O} hold.

  Clearly, \ref{assum:continuity:Y:g-theta} and
  \ref{assum:continuity:Y:phi-theta} hold by definitions of
  $\psi^\theta_y(\mathbf{x})$ and $g^\theta(\mathbf{x};y)$. It remains to show
  ~\ref{assum:practical-cond:O}. Since $\Theta$ is compact, then
\begin{equation*}
\underline{\omega}\leq \min_{1\leq\ell\leq d}\boldsymbol{\omega}_\ell ,
\; |\boldsymbol{\omega}|\leq\overline{\omega}, \,
\underline{b} \leq |\mathbf{b}| \leq\bar b,\,
\radius(\mathbf{A}+\mathbf{b}\boldsymbol{\gamma}^T)\le\bar{\rho}, \;
\left\|\mathbf{A}+\mathbf{b}\boldsymbol{\gamma}^T\right\|\le L
\end{equation*}
for some $(\underline{\omega},\overline{\omega},\underline{b},\overline{b},\bar \rho) \in (0,\infty)^4\times(0,1)$ and $L>0$.
By \cite[Lemma~12]{moulines-priouret-roueff-2005}, we note that this implies
that, for all $\bar\delta\in(\bar\rho,1)$, there exists $\bar C>0$ such that for all
$k\geq1$ and all $\theta\in\Theta$,
\begin{equation}
  \label{eq:uniform-spectral-raius}
  \left\|\left(\mathbf{A}+\mathbf{b}\boldsymbol{\gamma}^T\right)^k\right\|\leq
  \bar C\,\bar\delta^k\;.
\end{equation}
We set
$\Xset_1=[\underline \omega,\infty)^d\subset\Xset$ so that
\ref{assum:practical-cond:O}\ref{assum:exist:practical-cond:subX}
holds. Moreover, for all
$(\theta,\mathbf{x},y)\in\Theta\times\Xset_1\times\Yset$,
$g^\theta(\mathbf{x};y)\le(2\pi\underline\omega)^{-1/2}$. Thus,
Condition~\ref{assum:practical-cond:O}\ref{assum:momentCond} holds.  Now let
$x_1\in\Xset$.
Using~\eqref{eq:psi:nm-garch:iterate},~(\ref{eq:uniform-spectral-raius})
and~(\ref{eq:norm-control2}), we have, for all $\mathbf{x}\in\Xset$, $y_{1:n}\in\Yset^n$ and $\theta\in\Theta$,
\begin{align*}
\left|\f{\chunk{y}1n}(\mathbf{x}_1)-\f{\chunk{y}1n}(\mathbf{x})\right|&= \left|\mathbf{A}^n(\mathbf{x}_1-\mathbf{x})\right|\\
&\leq \bar C\,\bar\delta^n\left|\mathbf{x}_1-\mathbf{x}\right|\eqsp.
  \end{align*}
Using that the norm defining $\Xmet$ is equivalent to the norm $|\cdot|$, we get~\ref{assum:practical-cond:O}\ref{assum:exist:practical-cond:O} with
$$
\bar\psi(\mathbf{x})= \bar C'\,\left|\mathbf{x}_1-\mathbf{x}\right|\;,
$$
for some positive constant $\bar
C'$. Hence~\ref{assum:practical-cond:O}\ref{item:psibar} holds and since
$$
\left|\mathbf{x}_1-\psi^\theta_y(\mathbf{x}_1)\right|\leq (L+1)\,|\mathbf{x}_1|+\overline{\omega}+y^2\bar b\;,
$$
we also get~\ref{assum:practical-cond:O}\ref{item:barphi:bar:phi} provided that
\begin{equation}
  \label{eq:phibar-nmgarch-temp}
\bar\phi(y)\geq (L+1)\,|\mathbf{x}_1|+\overline{\omega}+y^2\bar b \;.
\end{equation}
It is straightforward to show that,
 for  all $\theta\in\Theta$, $\mathbf{x}\in\Xset_1$, $y\in\rset$, and $\ell\in\{1,\ldots,d\}$,
\begin{align*}
\left|\frac{\partial\ln g^\theta}{\partial x_\ell}(\mathbf{x};y)\right|&\le\frac{1}{2}\left(\frac{y^2}{\underline{\omega}^2}+\frac{1}{\underline{\omega}}\right)\eqsp.
\end{align*}
Thus, by the mean value theorem,  for all $\theta\in\Theta$, $(\mathbf{x},\mathbf{x}')\in\Xset_1\times\Xset_1$ and $y\in\Yset$,
\begin{align*}
\left|\ln g^\theta(\mathbf{x};y)-\ln g^\theta(\mathbf{x}';y)\right|&\le \frac{1}{2}\left(\frac{y^2}{\underline{\omega}^2}+\frac{1}{\underline{\omega}}\right)\,|\mathbf{x}-\mathbf{x}'|\eqsp.
\end{align*}
We thus obtain~\ref{assum:practical-cond:O}\ref{item:barphi:bar:phi},
\ref{assum:practical-cond:O}\ref{item:ineq:loggg}
and~\ref{assum:practical-cond:O}\ref{item:H} by setting $C=0$,

$$
H(u)=\sup_{d(\mathbf{x},\mathbf{x}')\leq u}|\mathbf{x}-\mathbf{x}'|\;,
$$
and
$$
\bar\phi(y)=(L+1)\,|\mathbf{x}_1|+\overline{\omega}+1/(2\underline{\omega})+y^2(\bar b+\underline{\omega}^2)\;.
$$
In addition, for all $\theta\in\Theta$ and $\mathbf{x}\in\Xset$, we have
$$
\int y^2 G^\theta(x,\rmd y) = \boldsymbol{\gamma}^T\mathbf{x} \;.
$$
Hence, using~(\ref{eq:Vbar-nmgarch}) with the above definitions, we
obtain~\ref{assum:practical-cond:O}\ref{item:22or23-barphi-V} and the proof is concluded.
\end{proof}


\subsection{The Threshold INGARCH model}

The threshold INGARCH$(1,1)$ in \autoref{example:tingarch} is a specific case of \autoref{def:obs-driv:gen} where $\nu$ is the counting measure on $\Yset=\zsetp$,
\begin{align}
&\psi^\theta_y(x)=\omega+ax+by\eqsp,\label{eq:tingarch:def:phi}\\
&g^\theta(x;y)
=\rme^{-(x\wedge \tau)} \frac{(x\wedge \tau)^y}{y!}\eqsp,\label{eq:tingarch:g}
\end{align}
with $\theta=(\omega, a, b, \tau)$ in a compact subset $\Theta$ of
$(0,\infty)^4$ and $\Xset=(0,\infty)$. In this model, if $a<1$, we then have
the ergodicity and consistency results as stated in
\autoref{theo:ergo-converge:tingarch} below.

\begin{theorem}\label{theo:ergo-converge:tingarch}
  Suppose that all $\theta=(\omega, a, b, \tau)$ in $\Theta$ satisfy
  $a<1$. Then
  Assumptions~\ref{assum:gen:identif:unique:pi}, \ref{ass:21-lyapunov},
  \ref{assum:continuity:Y:g-theta}, \ref{assum:continuity:Y:phi-theta} and \ref{assum:practical-cond:O} hold
  with $\bar V$ being defined as the identity function on $\Xset$ and with any $x_1\in\Xset$.
\end{theorem}

\begin{proof} As in the proofs of the two theorems above, for convenience, we
  divide the proof into two steps.

\noindent\textbf{Step 1.} We first prove
Assumptions~\ref{assum:gen:identif:unique:pi} and \ref{ass:21-lyapunov} by
applying \autoref{thm:ergodicity}. We set $\bar V(x)= V(x)=x$ and thus we only need to
check~~\ref{ass-ergo-Xset}, \ref{assum:weakFeller-V}, \ref{assum:weakFeller}, \ref{assum:reachable},
\ref{assum:bound:rho} and~\ref{assum:38--40:alpha-phi}. Condition~\ref{ass-ergo-Xset} holds with the usual metric on $\rset$. We have for all $\theta\in\Theta$,
$$
R V(x)=\omega + ax+b(x\wedge\tau) \le aV(x) + (\omega+b\tau),
$$
which yields  \ref{assum:weakFeller-V}.
The fact that the kernel $R$ is weak Feller easily
follows by observing that, as $x\to x'$, $\mathcal{P}(x)$ converges weakly
to $\mathcal{P}(x')$ and the map $x\mapsto x\wedge\tau$ is continuous, so \ref{assum:weakFeller} holds.

The proof of \ref{assum:reachable} is similar to the NBIN-GARCH case of
\autoref{theo:ergo-converge:nbigarch} and is thus omitted.
Assumption~\ref{assum:bound:rho} holds since we have for all
$(x,x',y)\in\Xset^2\times\Yset$ with $x\neq
x'$, $$\frac{|\psi_y(x)-\psi_y(x')|}{|x-x'|}=a<1\eqsp.
$$
To prove~\ref{assum:38--40:alpha-phi}, we apply ~\autoref{lem:det:alpha:gen}
with $\Cset=\Xset$, $\Sset=\{1\}$ (so $\mu$ boils down to the Dirac measure on $\{1\}$). For all
$(x,y)\in\Xset\times\Yset$, let $j(x)=\rme^{-(x\wedge\tau)}$ and
$h(x;y)=\frac{(x\wedge\tau)^y}{y!}$. Then $h$ indeed satisfies ~\ref{item:F1}. Thus
by~\autoref{lem:det:alpha:gen}, for all $(x,x')\in\Xset^2$ and $y\in\Yset$, we
get that
\begin{equation*}
\alpha(x, x')=\rme^{-(x\vee x')\wedge\tau+(x\wedge x')\wedge\tau}\in(0,1] \quad \text{and} \quad \phi(x,x')=x\wedge x'
\end{equation*}
satisfy ~\ref{assum:38--40:alpha-phi}\ref{item:alpha:phi:3}.

Let  $W(x,x')=1$ for all $(x,x')\in\Xset^2$, which is a constant function. Thus \ref{assum:38--40:alpha-phi}\ref{assum:definition-gamma-x} and \ref{assum:38--40:alpha-phi}\ref{assum:driftCond:W:new}  clearly hold. Moreover, \ref{assum:38--40:alpha-phi}\ref{assum:hyp:beta:gene} holds since
for all $(x, x') \in\Xset^2$, we have that
\begin{equation*}
1-\alpha(x,x')\leq x\vee x'-x\wedge x'=\lvert x-x'\rvert = W(x,x')\lvert x-x'\rvert\eqsp.
\end{equation*}
Therefore, \ref{assum:38--40:alpha-phi} holds, which completes \textbf{Step 1}.

\noindent\textbf{Step 2}. We now prove  
  \ref{assum:continuity:Y:g-theta}, \ref{assum:continuity:Y:phi-theta} and \ref{assum:practical-cond:O}.
By assumption on $\Theta$, then there exists $(\underline{\omega},\bar{\omega},\underline{b}, \bar b,\underline{\tau}, \bar \tau,\underline{\alpha}, \bar \alpha) \in (0,\infty)^6\times(0,1)^2$ such that
$$
\underline{\omega}\leq \omega \leq\bar \omega, \,\, \leq b \leq\bar{b}, \,\, \underline{\tau}\leq \tau\leq \bar{r}, \,\, \underline{\alpha}\leq a\leq \bar{\alpha}\eqsp.
$$
Clearly, \ref{assum:continuity:Y:g-theta} and
\ref{assum:continuity:Y:phi-theta} hold by definitions of $\psi^\theta_y(x)$
and $g^\theta(x;y)$.  It remains to check~\ref{assum:practical-cond:O} for a
well-chosen closed subset $\Xset_1$ and any $x_1\in\Xset$. Let
$\Xset_1=[\underline \omega,\infty)\subset\Xset$ so that
\ref{assum:practical-cond:O}\ref{assum:exist:practical-cond:subX} holds. By
noting that for all $(\theta,x,y)\in\Theta\times\Xset\times\Yset$,
$g^\theta(x;y)\le1$, we have
\ref{assum:practical-cond:O}\ref{assum:momentCond}.
From~\eqref{eq:notationItere:f:cl} and ~\eqref{eq:tingarch:def:phi}, we have
for all $s\leq t$, $y_{s:t}\in\Yset^{t-s+1}$, $x\in\Xset$ and
$\theta\in\Theta$,
\begin{equation}\label{eq:recurs:tingarch}
\f{\chunk{y}{s}{t}}(x)= \omega\left(\frac{1-a^{t-s+1}}{1-a}\right)+a^{t-s+1}x+b\sum_{j=0}^{t-s} a^jy_{t-j} \eqsp.
\end{equation}
Using~\eqref{eq:recurs:tingarch}, we have, for all $\theta\in\Theta$, $x\in\Xset$ and $y_{1:n}\in\Yset^n$,
\begin{equation*}
\left|\f{\chunk{y}1n}(x_1)-\f{\chunk{y}1n}(x)\right|= a^n\left|x_1-x\right|\le\bar{\alpha}^n\left|x_1-x\right|\eqsp.
  \end{equation*}
  This gives~\ref{assum:practical-cond:O}\ref{assum:exist:practical-cond:O}
  and~\ref{assum:practical-cond:O}\ref{item:psibar} by setting
  $\varrho=\bar{\alpha}<1$ and $\bar\psi(x)=|x_1-x|$.  Next we set
  $\bar\phi$, $H$ and $C$
  to meet Conditions~\ref{assum:practical-cond:O}\ref{item:barphi:bar:phi}
  and~\ref{assum:practical-cond:O}\ref{item:ineq:loggg} and~\ref{assum:practical-cond:O}\ref{item:H}. Let us write, for all
  $\theta\in\Theta$ and $y\in\Yset$,
$$
\left|x_1-\psi^\theta_y(x_1)\right|\le\omega+(1+a)x_1+by\le\bar\omega+(1+\bar{\alpha})x_1+\bar{b}y
$$
and, for all $(x,x')\in\Xset_1^2=[\underline \omega,\infty)^2$,
\begin{align*}
\left|\ln g^\theta (x;y)-\ln g^\theta (x';y)\right|
&=\left|\left(x'\wedge\tau-x\wedge\tau\right)+y\left(\ln(x\wedge\tau)-\ln(x'\wedge\tau)\right)\right|\\
&\leq \left(1+(\underline{\omega}\wedge\underline{\tau})^{-1}y\right)\;|x-x'| \;.
\end{align*}
Setting  $\bar\phi(y)=1+\bar\omega+(1+\bar{\alpha})x_1+\left(\bar{b}\vee(\underline{\omega}\wedge\underline{\tau})^{-1}\right)y$,
$H(x)=x$ and $C=0$ then yield Conditions~\ref{assum:practical-cond:O}\ref{item:barphi:bar:phi}, ~\ref{assum:practical-cond:O}\ref{item:ineq:loggg}
  and~\ref{assum:practical-cond:O}\ref{item:H}. Now
  \ref{assum:practical-cond:O}\ref{item:22or23-barphi-V} follows from
$$
\int \lnp y \; G^\theta(x,\rmd y) \leq \int y \; G^\theta(x,\rmd y) = x\wedge\tau\leq
\bar V(x)\;.
$$
This concludes the proof.
\end{proof}

\section{Numerical experiments}
\label{sec:numer-exper}
\subsection{Numerical procedure}\label{sec:deriv:cal}
In this part we provide a numerical method for computing the (conditional) MLE
$\mlY{x,n}$ for the parameter $\theta=(\omega, a, b, r)$ in the
NBIN-GARCH$(1,1)$ model introduced in \autoref{example:Nbigarch:defi} and
studied in \autoref{sec:nbin-garch-model}. It is convenient to write
$\theta=(\vartheta, r)$ with $\vartheta=(\omega, a, b)$ and then to write
$\psi^{\vartheta}_y(x)$ and $g^r(x;y)$ instead of $\psi^{\theta}_y(x)$ and
$g^\theta(x;y)$ in~(\ref{eq:Nbingarch:def:phi}) and~(\ref{eq:Nbingarch:g}),
respectively. In contrast to the approach used in \cite{zhu:2011}, we allow the
component $r$ to be any positive real number, rather than a discrete one and to
be unknown as well. We thus maximize jointly with respect to the parameters
$\vartheta$ and $r$ the log-likelihood function
$\lkdM[x,n]{\chunk{y}{1}{n}}[\theta]=\lkdM[x,n]{\chunk{y}{1}{n}}[(\vartheta,
r)]$. In practice, one does not rely on a compact set $\Theta$ of parameters
as in~\autoref{theo:ergo-converge:nbigarch}. Instead the maximization is
performed over all parameters $\omega>0$, $a>0$, $b>0$, $r>0$ such that  the stability
constraint $a+br<1$ holds (taken from~(\ref{eq:nbingarch-stability-cond-zhu})).
We use the constrained nonlinear optimization function \texttt{auglag}
(Augmented Lagrangian Minimization Algorithm) from the package \texttt{alabama}
(Augmented Lagrangian Adaptive Barrier Minimization Algorithm) in \texttt{R}. For this purpose we provide an initial parameter
point and a numerical computation of the normalized log-likelihood function
$\lkdM[x,n]{\chunk{y}{1}{n}}[\theta]$ and of its gradient. The initial point is
obtained  by applying a conditional least
square (CLS) estimation based on an ARMA$(1,1)$ representation of the model, see \cite[Section~3]{zhu:2011}. The
computation of the log-likelihood and of its
derivatives are derived as follows. For all $x\in\Xset$, denoting
$u_k^\vartheta=\f[\vartheta]{\chunk{y}{1}{k-1}} (x)$ for all $k\geq2$ and
$u_1^\vartheta=x$, we have
\begin{align*}
\lkdM[x,n]{\chunk{y}{1}{n}}[(\vartheta, r)]&= n^{-1} \sum_{k=1}^{n} \ln g^r\left(\f[\vartheta]{\chunk{y}{1}{k-1}} (x);y_k\right) \\
&=n^{-1}\ln g^r(x,y_1)+ n^{-1}\sum_{k=2}^{n} \ln g^r\left(u_k^\vartheta;y_k\right) \eqsp.
\end{align*}
The computation of $u_k^\vartheta$ for all $k\geq2$ is done iteratively by
observing that $u_k^\vartheta=\f[\vartheta]{y_{k-1}}(u_{k-1}^\vartheta)$ and the
computation of $\lkdM[x,n]{\chunk{y}{1}{n}}[(\vartheta, r)]$ is deduced.
The computation of the derivatives with respect to parameter $\theta=(\vartheta, r)$ of the
function $\lkdM[x,n]{\chunk{y}{1}{n}}[(\vartheta, r)]$ are then obtained in two
steps. First, for $k\geq2$, the derivative of $u_k^\vartheta$ with respect to $\vartheta$ are
obtained iteratively by ${\partial{u_{1}^\vartheta}}/{\partial{\vartheta}}=0$ and
\begin{equation*}
\frac{\partial{u_k^\vartheta}}{\partial{\vartheta}}=
(1, u_{k-1}^\vartheta, Y_{k-1})+a\frac{\partial{u_{k-1}^\vartheta}}{\partial{\vartheta}} \;.
\end{equation*}
Then the derivatives of $\lkdM[x,n]{\chunk{y}{1}{n}}[(\vartheta, r)]$ with respect to $\vartheta$ and $r$ are given by
\begin{equation*}
\frac{\partial{\lkdM[x,n]{}[(\vartheta,r)]}}{\partial{\vartheta}}=n^{-1}\sum_{k=1}^n\frac{\partial{\ln g^r}}{\partial{x}}\left(u_k^\vartheta;y_k\right)\frac{\partial{u_k^\vartheta}}{\partial{\vartheta}}=n^{-1}\sum_{k=2}^n\left(\frac{y_k}{u_k^\vartheta}-\frac{y_k+r}{1+u_k^\vartheta}\right)\frac{\partial{u_k^\vartheta}}{\partial{\vartheta}}
\end{equation*}
and
\begin{equation*}
\begin{split}
\frac{\partial{\lkdM[x,n]{}[(\vartheta,r)]}}{\partial{r}}&=n^{-1}\sum_{k=1}^n\frac{\partial{\ln g^r}}{\partial{r}}\left(u_k^\vartheta;y_k\right)\\
&=n^{-1}\sum_{k=1}^n\left(\Gamma_2(r+y_k)-\ln(1+u_k^\vartheta)\right)-\Gamma_2(r)\eqsp,
\end{split}
\end{equation*}
respectively, where $\Gamma_2$ is the digamma function $\Gamma_2(r)=\frac{\rmd}{\rmd{r}}{\ln\circ\Gamma(r)}$, $r>0$.

%

\subsection{Simulation study}
We consider two NBIN-GARCH$(1,1)$ models with parameters:
\begin{enumerate}
\item[(M.1)] $\thv=(\omega_\star, a_\star, b_\star, r_\star)=(3,\,.2,\,.2,\,2)$ and
\item[(M.2)] $\thv=(\omega_\star, a_\star, b_\star, r_\star)=(3,\,.35,\,.1,\,1.5)$.
\end{enumerate}
We simulated $m=200$ data sets for each sample size $n= 2^7, 2^8, 2^9$ and
$2^{10}$.  In \autoref{pic-difflog}, we display the obtained boxplots of the
difference of the normalized log-likelihood function evaluated respectively at
MLE and at the true value $\thv$. As predicted by the theory, this difference
appears to converge to $0$ as the number of observations $n\to\infty$. For the
NBIN-GARCH$(1,1)$ model, it can be shown that $\Theta_\star=\{\thv\}$, which
implies the convergence of the MLE to the true parameter. We can observe this
behavior for each component of the MLE for the two models in
\autoref{pic-mle:M1} and \autoref{pic-mle:M2}. We also report the Monte Carlo mean along with the mean
absolute deviation error (MADE):
$\text{MADE}=m^{-1}\sum_{j=1}^m|\hat\theta^j_{x,n}-\thv^j|$ as an evaluation
criterion for the estimated parameter in
\autoref{table:est}.


\begin{table}[H]
\caption{\label{table:est}Mean of estimates, MADEs (within parentheses) for the NBIN-GARCH$(1,1)$ models}
\begin{center}
\smallskip\noindent
\resizebox{\linewidth}{!}{
\begin{tabular}{r r r r r r l}
\cline{1-6}					
 &  & \multicolumn{4}{c}{Sample size $n$} \\ \cline{3-6}
Model & Parameter & $n=2^7$ & $n=2^8$ & $n=2^9$ & $n=2^{10}$ \\ \cline{1-6}
\multicolumn{1}{c}{\multirow{4}{*}{(M.1)} } &
\multicolumn{1}{c}{$\hat{\omega}$} & 3.311(.973) & 3.212(.719) & 3.108(.507) & 3.062(.372)      \\ 
\multicolumn{1}{c}{}                        &
\multicolumn{1}{c}{$\hat{a}$} & .165(.138) & .173(.113) & .187(.076) & .193(.055)     \\ 
\multicolumn{1}{c}{}                        &
\multicolumn{1}{c}{$\hat{b}$} & .194(.049) & .195(.034) & .197(.025) & .200(.018)      \\ 
\multicolumn{1}{c}{}                        &
\multicolumn{1}{c}{$\hat{r}$} & 2.045(.241) & 2.035(.166) & 2.020(.112) & 2.011(.074)      \\ \cline{1-6}

\multicolumn{1}{c}{\multirow{4}{*}{(M.2)} } &
\multicolumn{1}{c}{$\hat{\omega}$} & 3.525(1.325) & 3.362(1.258) & 3.326(1.041) & 3.167(.761)   \\ 
\multicolumn{1}{c}{}                        &
\multicolumn{1}{c}{$\hat{a}$} & .252(.227) & .290(.213) & .296(.170) & .319(.136)  \\ 
\multicolumn{1}{c}{}                        &
\multicolumn{1}{c}{$\hat{b}$} & .092(.056) & .097(.039) & .098(.028) & .100(.022)   \\ 
\multicolumn{1}{c}{}                        &
\multicolumn{1}{c}{$\hat{r}$} & 1.563(.175) & 1.539(.129) & 1.520(.093) & 1.513(.066)   \\ \cline{1-6}

\end{tabular}
}
\end{center}
\end{table}

\begin{figure}[H]
\centering
\includegraphics[width=\textwidth,scale=1,center]{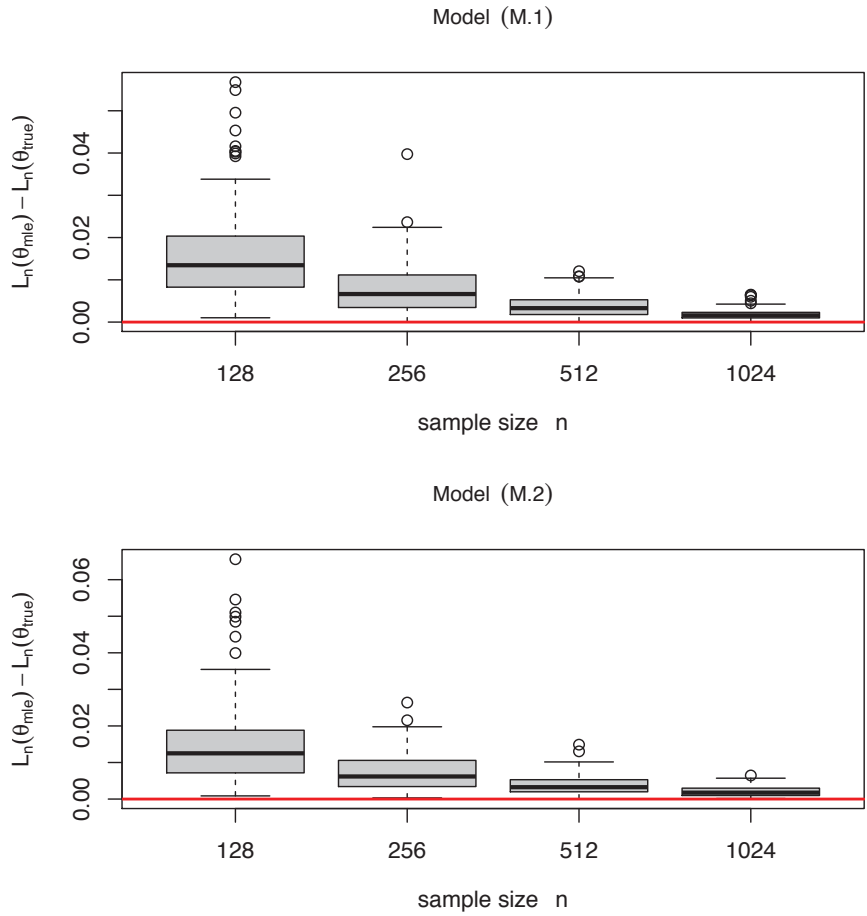}
\caption[BehaviorDifflogM1:M2]{Boxplots of the differences of log-likelihood functions evaluated at the estimated MLE and the true value for Models (M.1) and (M.2) with sample sizes $n=2^7, 2^8, 2^9$ and $n=2^{10}$, respectively. The red ``continuous" line indicates the position of zero.}
\label{pic-difflog}
\end{figure}

\begin{figure}[H]
\centering
\includegraphics[scale=1,center]{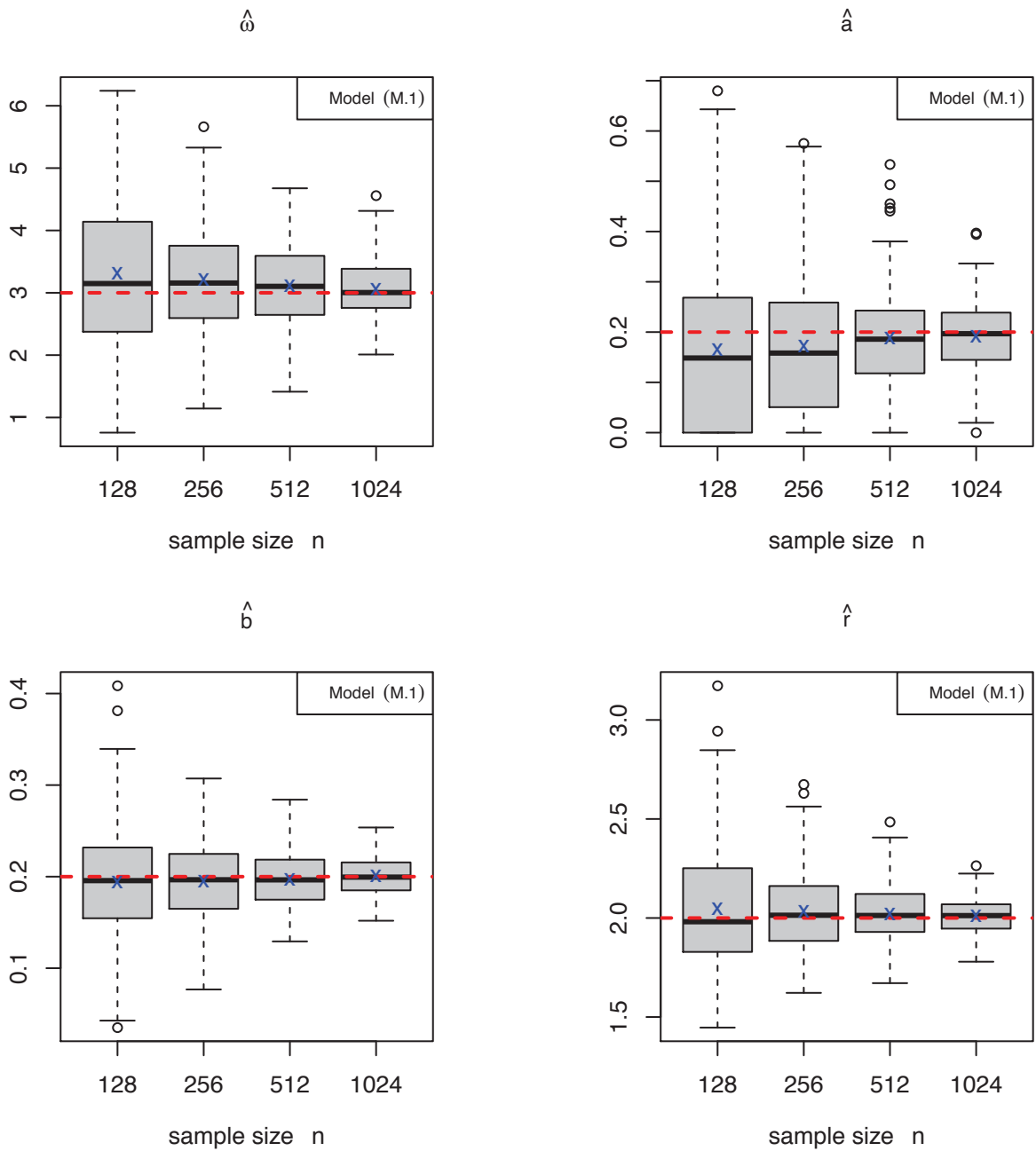}
\caption[BhaviorMLE:M1]{Boxplots of the estimated MLE for Model (M.1) with
  sample sizes $n=2^7, 2^8, 2^9$ and $n=2^{10}$, respectively. The red
  ``dashed'' line indicates the true value of the parameter and the blue
  ``\textsf{x}'' indicates the location of the Monte Carlo mean of the MLE.\label{pic-mle:M1}}
\end{figure}

\begin{figure}[H]
\centering
\includegraphics[scale=1,center]{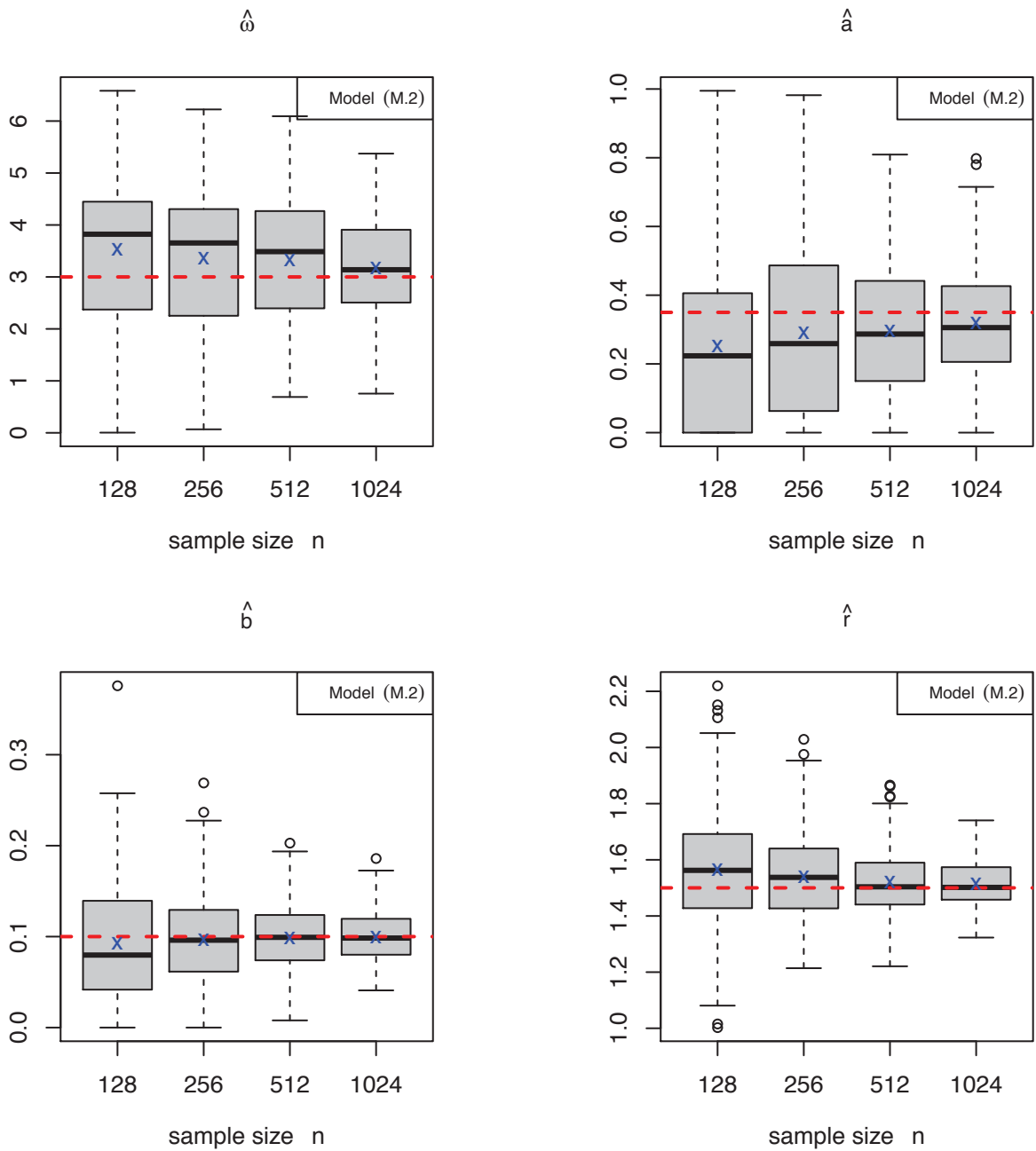}
\caption[BhaviorMLE:M2]{Same as \autoref{pic-mle:M1} but for Model (M.2).\label{pic-mle:M2}}
\end{figure}

\section{Postponed proofs}\label{sec:proofs}
\subsection{Convergence of the MLE}
\label{sec:convergence-mle}

Assumptions~\ref{assum:gen:identif:unique:pi} and~\ref{ass:21-lyapunov} are
supposed to hold throughout this section.
The proof of \autoref{thm:convergence-main} relies on the approach introduced
in \cite{pfanzagl:1969}, which was already used in \cite{dou:kou:mou:2013} for
a restricted class of observation-driven models.
Our main contribution here is
to provide the handy conditions listed in
Assumption~\ref{assum:practical-cond:O}. We first show that our conditions
imply~\ref{assum:Ptheta:thetastar:is:density} and the following one.
\begin{hyp}{B}
\item \label{assum:limit:Y} There exists $x_1\in\Xset$ such that, for all
  $\theta,\thv\in\Theta$, $p^{\theta}(Y_1\,|\,\chunk{Y}{-\infty}{0})$
  defined as in~(\ref{eq:def-p-theta-neq-thv}) with $x=x_1$ is finite $\tilde\PP^{\thv}\as$ Moreover, for all $\thv\in\Theta$, we
  have
\begin{align}\label{eq:unfi-limit-contrast-od}
\lim_{k \to \infty} \sup_{\theta \in \Theta} \left|\ln \frac{g^\theta (\f{\chunk{Y}{1}{k-1}}(x_1);Y_k)}{p^{\theta}(Y_k\,|\,\chunk{Y}{-\infty}{k-1})}\right| =0\quad\tilde\PP^{\thv}\as
\end{align}
\end{hyp}
Indeed we have the following lemma.
\begin{lemma}\label{lem:exist:practical-cond:O}
Assumptions \ref{assum:continuity:Y:g-theta}, \ref{assum:continuity:Y:phi-theta} and~\ref{assum:practical-cond:O}
imply~\ref{assum:limit:Y} and~\ref{assum:Ptheta:thetastar:is:density}.
\end{lemma}
\begin{proof}
See \ref{app:lem}.
\end{proof}
Now the proof of \autoref{thm:convergence-main} directly follows from the following
lemma.
\begin{lemma}\label{lemma:gen:convergence-od}
  Assume that
  \ref{assum:continuity:Y:g-theta},
  \ref{assum:continuity:Y:phi-theta} and  \ref{assum:practical-cond:O}\ref{assum:exist:practical-cond:subX}--\ref{assum:momentCond} hold and that $x_1$
  satisfies~\ref{assum:limit:Y}.  Then $\Theta_\star$ defined
  by~(\ref{eq:def-Theta-star-set}) is a non-empty closed subset of $\Theta$ and
  \eqref{eq:strong-consistency-prelim} holds.
\end{lemma}
\begin{proof}  By \cite[Theorem~33]{dou:kou:mou:2013}, to obtain~(\ref{eq:strong-consistency-prelim}), it is sufficient
  to show that, for all $\thv\in\Theta$, the two following assertions hold.
\begin{enumerate}[label=(\alph*)]
\item\label{item:finite:moment}
$\tilde\PE^{\thv} \left[\sup_{\theta\in\Theta}\lnp p^{\theta}(Y_1\,|\,\chunk{Y}{-\infty}{0})\right]<\infty$\,,
\item\label{item:con:limit}
the function $\theta\mapsto\ln p^{\theta}(Y_1\,|\,\chunk{Y}{-\infty}{0})$ is continuous on
$\Theta$, $\tilde\PP^{\thv}\as$
\end{enumerate}
In~\ref{assum:limit:Y}, $p^{\theta}(Y_1\,|\,\chunk{Y}{-\infty}{0})$ is
defined $\tilde\PP^{\thv}\as$ as the limit in~(\ref{eq:def-p-theta-neq-thv}) with $x=x_1$. So, $\tilde\PP^{\thv}\as$,
by~\ref{assum:practical-cond:O}\ref{assum:exist:practical-cond:subX}--\ref{assum:momentCond},
$p^{\theta}(Y_1\,|\,\chunk{Y}{-\infty}{0})$ is bounded by
the finite constant appearing in~\ref{assum:practical-cond:O}\ref{assum:momentCond}.
Hence Condition~\ref{item:finite:moment} holds.

Condition \ref{item:con:limit} then follows from~(\ref{eq:unfi-limit-contrast-od}).
Since almost sure convergence implies the convergence in probability and $\tilde\PP^{\thv}$ is shift invariant,
 the random sequence
 \begin{align*}
U_m\eqdef\sup_{\theta \in \Theta} \left|\ln \frac{g^{\theta} (\f{\chunk{Y}{-m}{0}}(x_1);Y_1)}{p^{\theta}(Y_1\,|\,\chunk{Y}{-\infty}{0})}\right|\;,\quad{m\in\zsetp}\;,
\end{align*}
converges to zero in $\tilde\PP^{\thv}$-probability. Then there exists a
subsequence of $(U_m)$ which converges $\tilde\PP^{\thv}\as$ to zero. Hence,
interpreting this convergence as a uniform (in $\theta$) convergence of $\ln
g^{\theta} (\f{\chunk{Y}{-m}{0}}(x_1);Y_1)$ to $\ln
{p^{\theta}(Y_1\,|\,\chunk{Y}{-\infty}{0})}$ to conclude
that~\ref{item:con:limit} holds, it is sufficient to show that
$\theta\mapsto\ln g^{\theta} (\f{\chunk{Y}{-m}{0}}(x_1);Y_1)$ is continuous for
all $m$ $\tilde\PP^{\thv}\as$ This is indeed the case by~\ref{assum:continuity:Y:g-theta}
and~\ref{assum:continuity:Y:phi-theta} and since $g^\theta(x;y)$
is positive.
\end{proof}

\subsection{Ergodicity}
\label{sec:ergodicity-proofs}

For proving \autoref{thm:ergodicity}, we first recall a more general set of conditions derived in
\cite{dou:kou:mou:2013}, which are based on the following definition.

\begin{definition}\label{def:coupling:con}
  Let $\bar G$ be a probability kernel from $\Xset^2$ to
  $\Ysigma^{\otimes 2}\otimes\mathcal{P}(\{0,1\})$ satisfying the following marginal conditions,
  for all $(x,x') \in \Xset^2$ and $B \in \Ysigma$,
\begin{equation}\label{eq:marginalConditionsH}
  \begin{cases}
\bar G((x,x');B \times \Yset\times\{0,1\})=G(x;B)\,, \\
\bar G((x,x');\Yset \times B\times\{0,1\} )=G(x';B) \eqsp,
  \end{cases}
\end{equation}
and such that the following coupling condition holds
\begin{equation}\label{eq:coupling-conditionGbar}
\bar G((x,x');\{(y,y)~:~y\in\Yset\}\times\{1\})=\bar G((x,x');\Yset^2\times\{1\})\;.
\end{equation}
Define the following quantities successively.
\begin{itemize}
\item The trace measure of $\bar G((x,x');\cdot)$ on the set $\{(y,y)~:~y\in\Yset\}\times\{1\}$ is
  denoted by
\begin{equation}\label{def:Gcheck}
\check G((x,x');B)=\bar G((x,x');\{(y,y)~:~y\in B\}\times\{1\}),\quad B\in\Ysigma\;.
\end{equation}
\item The probability kernel $\bar R$ from $(\Xset^2,\Xsigma^{\otimes 2})$ to
  $(\Xset^2\times\{0,1\},\Xsigma^{\otimes2}\otimes\mathcal{P}(\{0,1\})$ is
  defined for all $x,x'\in\Xset^2$ and $A\in \Xsigma^{\otimes2}$ by
\begin{equation}\label{def:Rbar}
\bar R((x,x');A\times\{1\})=\int_{\Yset}
\1_A(\psi_{y}(x),\psi_{y}(x'))\;\check G((x,x');\rmd y)\; .
\end{equation}
 \item The measurable function $\alpha$ from $\Xset^2$ to $[0,1]$ is defined by
 \begin{equation}\label{eq:expressionAlpha}
\alpha(x,x')=\bar R((x,x');\Xset^2 \times \{1\})
=\bar G((x,x');\Yset^2\times\{1\})\eqsp.
\end{equation}
\item The kernel $\hat R$ is defined for all $(x,x') \in \Xset^2$
  and $A \in \Xsigma^{\otimes 2}$ by
\begin{equation} \label{def:Rhat}
\hat R((x,x'); A)=
\begin{cases}
\displaystyle\frac{\bar R((x,x');A \times \{1\})}{\alpha(x,x')}
&\text{ if $\alpha(x,x')>0$,}\\  
0&\text{ otherwise.}
\end{cases}
 \end{equation}
\end{itemize}
\end{definition}
We can now introduce the so-called \emph{contracting condition} which yields ergodicity.
\begin{hyp}{A}
\item \label{assum:asympStrgFeller:general} There exists a kernel $\bar G$
  yielding $\alpha$ and $\hat R$ as in~\autoref{def:coupling:con}, a measurable
  function $W: \Xset^2 \to [1,\infty)$ satisfying
  Conditions~\ref{assum:38--40:alpha-phi}\ref{assum:definition-gamma-x} and~\ref{assum:38--40:alpha-phi}\ref{assum:hyp:beta:gene} and real
  numbers $(D,\zeta_1,\zeta_2,\rho) \in (\rset_+)^3 \times (0,1)$ such that for
  all $(x,x') \in \Xset^2$ and, for all $n\geq1$,
\begin{align}
& {\hat R}^n((x,x'); \Xmet) \leq D \rho^n \Xmet(x,x') \label{eq:majo:d}\;,\\
& {\hat R}^n((x,x'); \Xmet\times W) \leq D \rho^n \Xmet^{\zeta_1}(x,x')W^{\zeta_2}(x,x')\;. \label{eq:majo:d:w}
\end{align}
\end{hyp}
Under Conditions~\ref{ass-ergo-Xset}, \ref{assum:weakFeller-V},
\ref{assum:weakFeller}, \ref{assum:reachable} and
\ref{assum:asympStrgFeller:general} and by combining Theorem~6, Proposition~8
and Lemma~7 in~\cite{dou:kou:mou:2013}, we immediately obtain the following
result.
\begin{theorem} \label{thm:uniqueInvariantProbaMeasure}
  Assume~\ref{ass-ergo-Xset}, \ref{assum:weakFeller-V}, \ref{assum:weakFeller},
  \ref{assum:reachable} and \ref{assum:asympStrgFeller:general}. Then the Markov
  kernel $K$ admits a unique invariant distribution $\pi$ and $\pi_1(\bar
  V)<\infty$ for any $\bar V:\Xset\to\rsetp$ such that $\bar V\lesssim V$.
\end{theorem}

Assumptions ~\ref{ass-ergo-Xset}, \ref{assum:weakFeller-V},
\ref{assum:weakFeller} and \ref{assum:reachable} are quite usual and easy to
check.  The key point to obtain ergodicity is thus to construct $\bar G$
satisfying~\ref{assum:asympStrgFeller:general}. For this, we can also rely on
the following result which is quoted from \cite[Lemma~9]{dou:kou:mou:2013}.
\begin{lemma} \label{lem:practical:conditions}
Assume that there exists $(\rho, \beta) \in (0,1) \times \rset$ such that for all $(x,x')\in \Xset^2$,
\begin{align}
& \hat R\left((x,x'); \left\{(x_1,x_1')\in\Xset^2~:~\Xmet(x_1,x_1')>\rho\; \Xmet(x,x')\right\}\right)=0\, , \label{eq:contrac}\\
& \hat RW \leq W +\beta \;. \label{eq:driftCond:W}
\end{align}
Then, \eqref{eq:majo:d} and \eqref{eq:majo:d:w} hold.
\end{lemma}

Now we can prove that our set of conditions is sufficient.
\begin{proof}[Proof of \autoref{thm:ergodicity}]
We only need to show that  \ref{assum:bound:rho}
  and~\ref{assum:38--40:alpha-phi} imply ~\ref{assum:asympStrgFeller:general}. We preface our proof by the following lemma.

\begin{lemma}\label{lem:coupling:con}
Assume \ref{assum:38--40:alpha-phi}\ref{item:alpha:phi:3}.
Then one can define a kernel $\bar G$ as in~\autoref{def:coupling:con} with
the same $\alpha$ given in~\eqref{eq:expressionAlpha}.
Moreover, the kernel $\hat R$ defined by~(\ref{def:Rhat}) satisfies, for all $(x,x')\in\Xset^2$ such that
$\alpha(x,x')>0$ and all measurable functions $f:\Xset^2\to\rsetp$,
\begin{equation}\label{eq:rho-hatRef}
\hat R((x,x');f)= G(\phi(x,x');\tilde{f})\quad\text{with}\quad
\tilde{f}(y)=f(\psi_{y}(x),\psi_{y}(x'))\;.
\end{equation}
\end{lemma}
Let us conclude the proof of  \autoref{thm:ergodicity} before proving this
lemma. By \autoref{lem:coupling:con} and  \autoref{lem:practical:conditions}, it
remains to check that \eqref{eq:contrac} and \eqref{eq:driftCond:W} hold for all $(x,x')\in\Xset^2$.
 Observe that by definition of $\hat R$,
Condition~\ref{assum:38--40:alpha-phi}\ref{assum:driftCond:W:new} is equivalent to
$$
\sup_{(x,x')\in\Xset^2}\left(\hat RW(x,x')- W(x,x')\right) <\infty \;.
$$
so we can find $\beta\in\rset$ such
that~\eqref{eq:driftCond:W}  holds for all $(x,x')\in\Xset^2$.

Now, let $(x,x')\in\Xset^2$ and let $(X,X')$ be distributed according to $\hat{R}((x,x');\cdot)$ which is defined in~\eqref{eq:rho-hatRef}.
When $x=x'$, then $\Xmet(X,X')=0$, implying that
Condition~\eqref{eq:contrac} holds with any nonnegative $\rho$. For $x\neq x'$, let $\rho$ be defined by
\begin{equation}\label{eq:rho}
\rho=\sup_{\substack{(x,x',y)\in\Xset^2\times\Yset\\x\neq x'}}\frac{\Xmet(\psi_y(x),\psi_y(x'))}{\Xmet(x,x')}\;,
\end{equation}
which is in $(0,1)$ by~\ref{assum:bound:rho}. Then
\begin{equation*}
\frac{\Xmet(X,X')}{\Xmet(x,x')}=\frac{\Xmet(\psi_Y(x),\psi_Y(x'))}{\Xmet(x,x')}\le\rho\;.
\end{equation*}
Therefore, Condition~\eqref{eq:contrac} holds  for all $(x,x')\in\Xset^2$ with $\rho$ as
in~\eqref{eq:rho}.
\end{proof}
We conclude this section with the postponed
\begin{proof}[Proof of~\autoref{lem:coupling:con}]
Let $(x, x')\in\Xset^2$. We define $\bar G((x,x');\cdot)$ as the distribution
of $(Y, Y',U)$ drawn as follows.  We first draw a random variable $\bar Y$
taking values in $\Yset$ with density $g(\phi(x,x');\cdot)$ with respect to
$\nu$. Then we define $(Y,Y',U)$ by separating the two cases, 
$\alpha(x,x')=1$ and $\alpha(x,x')<1$.
  \begin{enumerate}[label=$\bullet$]
  \item Suppose that $\alpha(x,x')=1$. Then
    from~\ref{assum:38--40:alpha-phi}\ref{item:alpha:phi:3}, we have
 $$
 G(x;\cdot)=G(x';\cdot)=G(\phi(x,x');\cdot)\eqsp.
 $$
 In this case, we set $(Y,Y',U)=(\bar{Y},\bar Y,1)$.
\item Suppose now that $\alpha(x,x')<1$. Then,
 using~(\ref{eq:con:min-g-phi}), the functions
$$
(1-\alpha({x,x'}))^{-1}\left[g(x;\cdot)-\alpha({x,x'})g(\phi(x,x');\cdot)\right]
$$
and
$$(1-\alpha({x,x'}))^{-1}\left[g(x';\cdot)-\alpha({x,x'})g(\phi(x,x');\cdot)\right]\eqsp,$$
are probability density functions with respect to $\nu$ and we let
$\Lambda$ and $\Lambda'$ be two independent random variables taking values
 in $\Yset$ drawn with these two density functions,
respectively. In this case we draw $U$ independently according to a Bernoulli variable with
mean $\alpha(x, x')$ and set
\[
(Y,Y')=
\begin{cases} (\bar Y, \bar Y)  \quad &\text{if $U=1$} \eqsp,\\
(\Lambda, \Lambda') \quad &\text{if $U=0$} \eqsp.
\end{cases}
\]
  \end{enumerate}
One can easily check that the so defined kernel $\bar G$ satisfies ~\eqref{eq:marginalConditionsH} and~\eqref{eq:coupling-conditionGbar}. Moreover, for all $(x,x')\in\Xset^2$,
$$
\bar G((x,x');\Yset^2\times\{1\})=\PP(U=1)=
  \alpha(x,x')\;,
$$
which
is compatible with~(\ref{eq:expressionAlpha}).
The kernel $\hat R$ is defined by setting $\hat R((x,x');\cdot)$ as the
conditional distribution of $(X,X')=(\psi_Y(x),\psi_Y(x'))$ given that
$U=1$.
To complete the proof of~\autoref{lem:coupling:con}, observe that for any
measurable $f:\Xset^2\to\rsetp$, we have, for all $(x,x')\in\Xset^2$ such that $\alpha(x,x')>0$,
\begin{align*}
\hat R((x,x');f)&=\PE\left[f(\psi_Y(x),\psi_{Y}(x'))\mid U=1\right] \\
            &=\PE\left[f(\psi_{\bar Y}(x),\psi_{\bar Y}(x'))\right]\\
            &= G(\phi(x,x');\tilde{f})\;,
\end{align*}
where $\tilde{f}(y)= f(\psi_y(x),\psi_y(x'))$ for all $y\in\Yset$.

\end{proof}

\subsection{Proof of~\autoref{lem:exist:practical-cond:O}}\label{app:lem}
Under~\ref{ass:21-lyapunov},
Assumptions~\ref{assum:practical-cond:O}\ref{item:22or23-barphi-V} implies that for all $\theta\in\Theta$,
\begin{equation}
\label{eq:bar:phi:mc}
\pi^\theta_2\left(\lnp (\bar{\phi})\right)<\infty\;,
\end{equation}
and if moreover $C>0$,
\begin{equation}
\label{eq:bar:phi:strong:mc}  \pi^\theta_2\left(\bar{\phi}\right)<\infty \;.
\end{equation}
For proving~\autoref{lem:exist:practical-cond:O}, we will also make use of
\cite[Lemma~34]{dou:kou:mou:2013} which we quote here for convenience.
\begin{lemma} \label{lem:useful}
Let \sequence{U}[n][\zsetp] be a stationary sequence of real-valued random variables on $(\Omega, \mcf, \PP)$. Assume that $\PE{(\lnp |U_0|)}<\infty$. Then, for all $\eta \in (0,1)$,
$$
\lim_{k \to \infty} \eta^k U_k =0\, , \quad \PP\as
$$
\end{lemma}
\begin{proof}[Proof of~\autoref{lem:exist:practical-cond:O}]
  We first show that $p^{\theta}(y|\chunk{Y}{-\infty}{0})$
  in~(\ref{eq:def-p-theta-neq-thv}) is finite for $x=x_1$ $\tilde \PP^{\thv}\as$\ By~\ref{assum:continuity:Y:g-theta}, this follows
by writing
\begin{equation}\label{eq:p-theta-thv:def}
p^{\theta}\left(y_1\,|\,\chunk{y}{-\infty}{0}\right) =
g^\theta\left(\f[\theta]{\chunk{y}{-\infty}{0}};y_1\right)\eqsp,
\end{equation}
if, for all $\theta,\thv\in\Theta$, the limit
\begin{equation}
  \label{eq::exist:x0:O}
\f[\theta]{\chunk{Y}{-\infty}{0}}=\lim_{m \to \infty}\f{\chunk{Y}{-m}{0}}(x_1)
\quad\text{ is well defined}\quad \tilde \PP^{\thv}\as
\end{equation}
For all $\theta\in\Theta$, $m\geq 0$,
  $x\in\Xset$ and $\chunk{y}{-m}{0}\in\Yset^{m+1}$,
  using~\ref{assum:practical-cond:O}\ref{assum:exist:practical-cond:O}, we have
\begin{equation} \label{eq:diff:x:y}
\Xmet(\f[\theta]{\chunk{y}{-m}{0}}(x_1),\f[\theta]{\chunk{y}{-m}{0}}(x)) \leq \varrho^{m+1} \; \bar\psi(x)\eqsp.
\end{equation}
Taking $x=\psi^\theta_{y_{-m-1}}(x_1)$ and
using~\ref{assum:practical-cond:O}\ref{item:barphi:bar:phi}, we obtain, for
all  $\chunk{y}{-m-1}{0}\in\Yset^{m+2}$,
\begin{equation*}
\Xmet(\f[\theta]{\chunk{y}{-m}{0}}(x_1),\f[\theta]{\chunk{y}{-m-1}{0}}(x_1)) \leq \varrho^{m+1} \; \bar\phi\left(y_{-m-1}\right)\eqsp.
\end{equation*}
Using~(\ref{eq:bar:phi:mc}) and
\autoref{lem:useful}, we have that
\begin{equation}
  \label{eq:phibar-infty}
\forall \eta\in(0,1),\quad\sum_{k\in\zset}\eta^{|k|}  \bar\phi\left(Y_{k}\right) < \infty\eqsp,\quad\tilde\PP^{\thv}\as\eqsp,
\end{equation}
and thus $\left(\f[\theta]{\chunk{Y}{-m}{0}}(x_1)\right)_{m\geq0}$ is a Cauchy
sequence $\tilde\PP^{\thv}\as$ Its limit exists $\tilde\PP^{\thv}\as$, since
$(\Xset,d)$ is assumed to be complete, which defines the $\Xset$-valued
random variable $\f[\theta]{\chunk{Y}{-\infty}{0}}$ for all
$\theta,\thv\in\Theta$ when $Y$ has distribution  $\tilde\PP^{\thv}\as$ Thus~(\ref{eq::exist:x0:O}) holds and we further obtain that
\begin{align}\nonumber
 \sup_{\theta\in\Theta}
 \Xmet(\f[\theta]{\chunk{Y}{-k}{0}}(x_1),x_1)
&\leq  \sup_{\theta\in\Theta}\sum_{m=0}^{k}
 \Xmet(\f[\theta]{\chunk{Y}{-m}{0}}(x_1),\f[\theta]{\chunk{Y}{-m+1}{0}}(x_1))\\
  \label{eq:limit-bound-cauchy-beforelimit}
&\leq\sum_{m\ge0}\varrho^{m} \;
 \bar\phi\left(Y_{-m}\right)<\infty\eqsp, \quad\tilde\PP^{\thv}\as
\end{align}
so that, letting $k\to\infty$,
\begin{align}
  \label{eq:limit-bound-cauchy}
 \sup_{\theta\in\Theta}
 \Xmet(\f[\theta]{\chunk{Y}{-\infty}{0}},x_1)
\leq\sum_{m\ge0}\varrho^{m} \;
 \bar\phi\left(Y_{-m}\right)<\infty\eqsp, \quad\tilde\PP^{\thv}\as
\end{align}
Let us now
prove~\ref{assum:Ptheta:thetastar:is:density}. Relation~(\ref{eq:p-theta-thv:def})
directly
yields~\ref{assum:Ptheta:thetastar:is:density}\ref{item:PthetaNEQthetastar:is:density}.
Let us
prove~\ref{assum:Ptheta:thetastar:is:density}\ref{item:PthetaISthetastar:is:density},
hence consider the case $\theta=\thv$. Using~\eqref{eq:diff:x:y}, we have
\begin{equation*}
\Xmet(\f[\thv]{\chunk{Y}{-m}{0}}(x_1),\f[\thv]{\chunk{Y}{-m}{0}}(X_{-m})) \leq \varrho^{m+1} \; \bar\psi(X_{-m})\quad\PP^{\thv}\as
\end{equation*}
Since $\{\bar\psi(X_{-m})\}_{m\geq0}$ is stationary under $\PP^\thv$,
it is bounded in probability, and since $\varrho<1$, for all $\epsilon>0$, we have
\begin{equation}
  \label{eq:proba-limit}
 \lim_{m\to\infty}\PP^{\thv}\left(\Xmet\left(\f[\thv]{\chunk{Y}{-m}{0}}(X_{-m}),\f[\thv]{\chunk{Y}{-m}{0}}(x)\right)>\epsilon\right)=0\eqsp.
\end{equation}
Note that for all $m\geq1$, $\f[\thv]{\chunk{Y}{-m}{0}}(X_{-m})=X_1$ $\PP^{\thv}\as$, hence
we get that
\begin{equation}
  \label{eq:eq-ftheta'theta'X1}
\f[\thv]{\chunk{Y}{-\infty}{0}}= X_1 \quad  \PP^{\thv}\as
\end{equation}
To complete the proof of~\ref{assum:Ptheta:thetastar:is:density}\ref{item:PthetaISthetastar:is:density}, we need to show that, under $\tilde\PP^{\thv}$, $y\mapsto g^\thv
(\f[\thv]{\chunk{Y}{-\infty}{0}};y)=g^\thv
(X_1;y)$ is the conditional density of
$Y_1$ given $\chunk{Y}{-\infty}{0}$, that is, for any $B\in\Ysigma$,
$$
\int\1_{B}(y)g^{\thv}\left(X_1;y\right)\nu(\rmd
y)=\PP^{\thv}\left(Y_1\in B\,|\,\chunk{Y}{-\infty}{0}\right) \;.
$$
Now, note that, by defintion of $\PP^{\thv}$,
$$
\int\1_{B}(y)g^\thv\left(X_1;y\right)\nu(\rmd y) = \PP^{\thv}
\left(Y_1\in B\mid X_1\right) =\PP^{\thv} \left(Y_1\in B\mid X_1,\chunk{Y}{-\infty}{0}\right)\;.
$$
But since~(\ref{eq:eq-ftheta'theta'X1}) implies that $X_1$ is
$\sigma(\chunk{Y}{-\infty}{0})$-measurable, $X_1$ can be removed in the last
conditioning, which concludes the proof~\ref{assum:Ptheta:thetastar:is:density}\ref{item:PthetaISthetastar:is:density}.

Finally, it remains to show the uniform
convergence~(\ref{eq:unfi-limit-contrast-od}) in~\ref{assum:limit:Y}.
By~\ref{assum:continuity:Y:phi-theta} and~(\ref{eq::exist:x0:O}), we have, for
all $\theta,\thv\in\Theta$, $k\in\zsetp$,
\begin{equation}\label{eq:com:psi:inf}
\f[\theta]{\chunk{Y}{-\infty}{k-1}}=\f[\theta]{\chunk{Y}{1}{k-1}}\left(\f[\theta]{\chunk{Y}{-\infty}{0}}\right)\eqsp,\quad\tilde\PP^{\thv}\as
\end{equation}
From~\ref{assum:practical-cond:O}\ref{assum:exist:practical-cond:O} and
\eqref{eq:com:psi:inf}, we get
$$
\Xmet(\f[\theta]{\chunk{Y}{1}{k-1}}(x_1),\f[\theta]{\chunk{Y}{-\infty}{k-1}})
\le\varrho^{k-1} \bar\psi\left(\f[\theta]{\chunk{Y}{-\infty}{0}}\right)\eqsp,\quad\tilde\PP^{\thv} \as
$$
On the other hand~\ref{assum:practical-cond:O}\ref{item:psibar}
and~(\ref{eq:limit-bound-cauchy}) imply
\begin{equation}\label{eq:finite:sup}
\sup_{\theta\in\Theta}\bar\psi\left(\f[\theta]{\chunk{Y}{-\infty}{0}}\right)<\infty\eqsp,\quad\tilde\PP^{\thv}\as\eqsp,
\end{equation}
which, with the previous display, yields,
\begin{equation}\label{eq:A:sup:theta}
\sup_{\theta\in\Theta}
\Xmet(\f[\theta]{\chunk{Y}{1}{k-1}}(x_1),\f[\theta]{\chunk{Y}{-\infty}{k-1}})
=O_{k\to\infty}\left(\varrho^{k}\right)\quad\tilde\PP^{\thv}\as
\end{equation}
Since $\Xset_1$ is closed and satisfies
Condition~\ref{assum:practical-cond:O}\ref{assum:exist:practical-cond:subX},
we have that, $\f{\chunk{Y}{1}{k-1}}(x_1)$ and
$\f[\theta]{\chunk{Y}{-\infty}{k-1}}$ are in $\Xset_1$ for all $k\geq2$. Thus
Condition~\ref{assum:practical-cond:O}\ref{item:ineq:loggg} gives that
$$
\sup_{\theta \in \Theta} \left|\ln \frac{g^\theta
    (\f{\chunk{Y}{1}{k-1}}(x_1);Y_k)}{g^\theta
    (\f[\theta]{\chunk{Y}{-\infty}{k-1}};Y_k)}\right|\leq A_k(1)\times
A_k(2)\times A_k(3)\times A_k(4)\quad \tilde\PP^{\thv}\as\;,
$$
where
\begin{align*}
A_k(1)&=\sup_{\theta \in \Theta}
H\left(\Xmet(\f{\chunk{Y}{1}{k-1}}(x_1),\f[\theta]{\chunk{Y}{-\infty}{k-1}})\right)\\
A_k(2)&=\sup_{\theta \in \Theta}
\rme^{C\,\Xmet\left(x_1,\f[\theta]{\chunk{Y}{-\infty}{k-1}}\right)}\\
A_k(3)&=\sup_{\theta \in \Theta}
\rme^{C\,\Xmet\left(x_1,\f{\chunk{Y}{1}{k-1}}(x_1)\right)}\\
A_k(4)&=\bar\phi(Y_k) \;.
\end{align*}
By~\eqref{eq:A:sup:theta} and~\ref{assum:practical-cond:O}\ref{item:H}, we
have
$$
A_k(1)=O_{k\to\infty}\left(\varrho^{k}\right)\quad\tilde\PP^{\thv}\as
$$
With~(\ref{eq:phibar-infty}), this yields~(\ref{eq:unfi-limit-contrast-od}) in the case where $C=0$.
For $C>0$, we further observe that, by~\eqref{eq:limit-bound-cauchy} and~(\ref{eq:bar:phi:strong:mc}), we have, for all $\thv\in\Theta$ and $k\in\zsetp$,
\begin{align*}
\tilde\PE^{\thv}\left[\lnp A_k(2) \right]
\le\tilde\PE^{\thv}\left[C\sum_{m\ge0}^\infty\varrho^m\bar\phi\left(Y_{-m+k-1}\right)\right]=\frac{C\pi_2^{\thv}\left(\bar\phi\right)}{1-\varrho}<\infty\eqsp.
\end{align*}
Then~\autoref{lem:useful} implies that, $\tilde\PP^{\thv}\as$,
$A_k(2)=O(\eta^{-k})$ for any $\eta\in(0,1)$. The same property applies
similarly to $A_k(3)$ by using~(\ref{eq:limit-bound-cauchy-beforelimit}) in
place of~(\ref{eq:limit-bound-cauchy}). This yields~(\ref{eq:unfi-limit-contrast-od}) in
the case where $C>0$, which concludes the proof.
\end{proof}


\section*{Acknowledgement}
We are thankful to the editor-in-charge and anonymous referee for the insightful comments and the helpful suggestions that lead to improve this paper.

\bibliography{monybibs_Hal}

\end{document}